\documentclass[12pt]{amsart}
\usepackage{amssymb,latexsym,eucal}

\usepackage{hyperref}
\usepackage[letterpaper,centering,text={6.5in,9in}]{geometry}
\usepackage{graphicx}
\usepackage{breqn}
\usepackage{bbold}

\usepackage{latexsym}
\usepackage[utf8]{inputenc}
\usepackage{amsmath}
\usepackage{amsthm}
\usepackage{amsfonts}
\usepackage{amssymb}
\usepackage{epsfig}
\usepackage{float}
\usepackage{nicefrac}
\usepackage[all]{xy}
\usepackage{graphicx}
\usepackage{enumerate}
\usepackage{xcolor}
\usepackage{mathtools}    
\DeclarePairedDelimiterX\setc[2]{\{}{\}}{\,#1 \;\delimsize\vert\; #2\,}
\usepackage{graphicx,accents}

\makeatletter
\DeclareRobustCommand{\cev}[1]{%
  \mathpalette\do@cev{#1}%
}
\newcommand{\do@cev}[2]{%
  \fix@cev{#1}{+}%
  \reflectbox{$\m@th#1\vec{\reflectbox{$\fix@cev{#1}{-}\m@th#1#2\fix@cev{#1}{+}$}}$}%
  \fix@cev{#1}{-}%
}
\newcommand{\fix@cev}[2]{%
  \ifx#1\displaystyle
    \mkern#23mu
  \else
    \ifx#1\textstyle
      \mkern#23mu
    \else
      \ifx#1\scriptstyle
        \mkern#22mu
      \else
        \mkern#22mu
      \fi
    \fi
  \fi
}

\makeatother

\newtheorem{theorem}{Theorem}[section]
\newtheorem{corollary}[theorem]{Corollary}
\newtheorem{proposition}[theorem]{Proposition}
\newtheorem{lemma}[theorem]{Lemma}

\newtheorem{question}[theorem]{Question}

\theoremstyle{definition}

\newtheorem{remark}[theorem]{Remark}
\newtheorem{example}[theorem]{Example}

\newcommand{\CC}{{\mathbb C}}

\newcommand{\RR}{{\mathbb R}}
\newcommand{\ZZ}{{\mathbb Z}}
\newcommand{\NN}{{\mathbb N}}

\begin{document}

\title[On symplectic capacities and their blind spots]{On symplectic capacities and their blind spots}

\author{Ely Kerman}
\author{Yuanpu Liang}
\address{Department of Mathematics\\
University of Illinois at Urbana-Champaign\\
1409 West Green Street\\
Urbana, IL 61801, USA.}
\thanks{The first named author is supported by a grant from the Simons Foundation}

\date{\today}

\begin{abstract} In this paper we settle three basic questions concerning the 
Gutt-Hutchings capacities from \cite{gh} which are conjecturally equal to the  Ekeland-Hofer capacities from \cite{eh1,eh2}. Our primary result settles a version of the recognition question from \cite{chls}, in the negative. We prove that the Gutt-Hutchings capacities together with the volume, do not constitute a complete set of symplectic invariants for star-shaped domains with smooth boundary. In particular, we construct a smooth family of such domains, all of which have the same Gutt-Hutchings capacities and volume, but no two of which are symplectomorphic. 

We also establish two independence properties of the Gutt-Hutchings capacities. We prove that, even for star-shaped domains with smooth boundaries, these capacities are independent from the volume by constructing a family of star-shaped domains with smooth boundaries,  whose Gutt-Hutchings capacities all agree but whose volumes all differ. We also prove that the capacities are mutually independent  by constructing, for any $j \in \NN$, a family of star-shaped domains, with smooth boundary and the same volume, whose capacities are all equal but the $j^{th}$. 

The constructions underlying these results are not exotic. They are convex and concave toric domains as defined  in \cite{gh}, where the authors also establish beautiful formulae for their capacities.  A key to the progress made here is a significant simplification of these formulae under an additional symmetry assumption. This simplification allows us to identify new blind spots of the Gutt-Hutchings capacities which are used to construct the desired examples. This simplification also yields explicit computations of the Gutt-Hutchings capacities in many new examples. 
%More precisely we define a class of smooth convex domains and deformations within this class which increase volume while fixing all capacity carriers and their actions.

%In particular, $U$ is a proper subset of $V$ and both are convex toric domains in the sense of \cite{gh}. 
\end{abstract}

\maketitle

\section{Introduction}

In the papers \cite{eh1} and \cite{eh2}, Ekeland and Hofer introduce the formal notion of a symplectic capacity  and construct a sequence of capacities, $\{c^{\mathrm{EH}}_k\}_{k \in \NN}$,  for subsets of $(\RR^{2n}, \omega)$. These rich symplectic invariants are defined in terms of the closed orbits of autonomous Hamiltonian flows. They are difficult to compute and their values are still only known completely for model subsets like symplectic ellipsoids and polydisks. In  \cite{gh},  Gutt and Hutchings use $S^1$-equivariant Floer theory to construct another sequence of symplectic capacities, $\{c_k\}_{k \in \NN}$, for star-shaped domains in  $(\RR^{2n}, \omega)$.\footnote{Adopting the convention of \cite{gh}, a {\em domain} here will refer to the closure of an open subset of a Euclidean space.} These are also defined in terms of closed orbits of Hamiltonian flows and are conjectured to be equal to the Ekeland-Hofer capacities. In \cite{gh}, the authors also derive combinatorial formulae for their capacities for both convex and concave toric domains. Here we establish a significant simplification of these formulae in the presence of an additional symmetry. This simplification reveals new blind spots of the capacities which allow us to settle several basic questions concerning their properties.  
%This is the content of Theorems \ref{mutual}, \ref{novolume}, and \ref{blind}, below.

\subsection{Results} Motivation for the questions addressed in this paper can already be found in the simplest nontrivial computations of the Gutt-Hutchings capacities. Fix a number $a>1$. The $k$th capacity of the symplectic  ellipsoid 
\begin{equation*}
\label{ }
E(1,a) =\left\{ (z_1, z_2) \in \CC^2 \mid \pi|z_1|^2 +\frac{\pi |z_2|^2}{a} \leq 1 \right\}
\end{equation*} 
is the $k$th  element of the sequence obtained by ordering the set $\{\NN \cup a\NN\}$ in nonincreasing order with repetitions. In {\em meta-code}, 
\begin{equation}
\label{ellipsoid}
c_k(E(1,a)) =  \left(\mathrm{Sort}[\{\NN \cup a\NN\}]\right)[[k]].
\end{equation}
The $k$th capacity of the polydisk
\begin{equation*}
\label{poly}
P(1,a) =\left\{ (z_1, z_2) \in \CC^2 \mid \pi|z_1|^2 \leq 1, \pi |z_2|^2 \leq a \right\}. 
\end{equation*}
 is 
\begin{equation}
\label{polyd}
c_k(P(1,a)) = k.
\end{equation}
Formulas \eqref{ellipsoid} and \eqref{polyd} are established in Section $2$ of \cite{gh}, and agree with the corresponding formulas for the Ekeland-Hofer capacities from Section III of \cite{eh2}.
 
\bigskip

\noindent{\bf How do the $c_k$ develop blind spots at corners?}

\medskip

It follows from \eqref{ellipsoid} that the collection of capacities $\{c_k(E(1,a))\}$ sees the defining parameter $a$. In contrast,  it follows from  \eqref{polyd} that  for $P(1,a)$ the parameter $a$ is invisible to the capacities. It \emph{hides} from them in the corners of the boundary of $P(1,a)$.  In Section \ref{forgetting}, we investigate the formation of these well-known blind spots of capacities at corners by analyzing the capacities of the domains 
\begin{equation}
\label{pellipse}
E_p(1,a) = \left\{ (z_1, z_2) \in \CC^2 \mid  (\pi |z_1|^2)^p + \left(\frac{\pi |z_2|^2}{a}\right)^p \leq 1\right\}.
\end{equation}
As $p$ goes from $1$ to $\infty$, these domains connect the ellipsoid $E(1,p)$ to the polydisk $P(1,a)$.  By refining the formulas from \cite{gh}, we give an explicit description of the process by which the capacities $c_k(E_p(1,a))$ loose sight of $a$ along the way. In short, for a fixed $k$ this blind spot develops in an instant. Above a specific value of $p$  the capacity $c_k(E_p(1,a))$ ceases to depend on $a$ (see Lemma \ref{k}). On the other hand, for each $p \in [1, \infty)$ there are infinitely many $c_k(E_p(1,a))$ which depend nontrivially on $a$ (see Lemma \ref{p}). In other words, the set of capacities $\{c_k(E_p(1,a))\}$ only looses sight of  $a$ at infinity.

\bigskip

\noindent{\bf Can the $c_k$ detect the volume of domains with smooth boundaries?}

\medskip

It is clear from formula \eqref{polyd} that the Gutt-Hutchings capacities do not see volume in general and, even worse, can fail to detect symplectic factors in symplectic product manifolds.\footnote{For $c_1$ this also occurs, over a limited range, in locally trivial symplectic bundles, \cite{ke}, \cite{lu}.}  
%This phenomenon is one consequence of the manner in which the capacities behave on general products.
%\begin{proposition}\label{product}(\cite{chls})
%If the domains $U \subset \RR^{2n}$ and $V \subset \RR^{2m}$ are star-shaped, then 
%\begin{equation*}
%\label{ }
%c^{\mathrm{EH}}_k(U \times V) = \min_{i+j=k} \{c^{\mathrm{EH}}_i(U)+ c^{\mathrm{EH}}_j(V)\}
%\end{equation*}
%where $i$ and $j$ are nonnegative integers and $c_0^{\mathrm{EH}} =0$.
%\end{proposition}
One can preclude product domains by restricting to the case of domains with smooth boundary. In this setting, the analysis of the capacities $c_k(E_p(1,a))$ in Section \ref{forgetting} (in particular Lemma \ref{p}) suggests that it might be possible for the set of capacities to detect volume (which is determined by $a$ in that setting). We show that this is a false hope.

\begin{theorem}\label{novolume}
There is a smooth family $V_{\delta}$ of toric star-shaped domains in $\RR^4$ with smooth boundary such that $\delta \mapsto c_k(V_{\delta})$ is constant for all $k$ and $\mathrm{volume}(V_{\delta}) = \mathrm{volume}(V_0) + \delta.$
\end{theorem}

\bigskip

\noindent{\bf Are the $c_k$ mutually independent?} 

\medskip

Within the symplectic category  of ellipsoids or  polydisks there are no variations which vary one capacity, or even finitely many capacities, and leave the others fixed. It is not immediately clear whether this interdependence  is a general feature of the Gutt-Hutchings capacities or an exception corresponding to the simplicity of the closed characteristics on the boundaries of ellipsoids and polydisks. For a generic star-shaped domain $U$ with a smooth boundary, $\partial U$,  there are infinitely many geometrically distinct closed characteristics on $\partial U$ and they are all nondegenerate. In this generic setting, each $c_k(U)$ is equal to a multiple of the symplectic action of one of the closed characteristics, \cite{gh}. It might be that the $c_k(U)$ always correspond to multiples of the actions of only finitely many characteristics, and that any variation of $U$ that alters the actions of one of these characteristics must then alter infinitely many capacities. Our second result shows that this is not the case. 

\begin{theorem}\label{mutual}
For every  $j \in \NN$, there is a smooth family $V^j_{\delta}$ of toric star-shaped domains in $\RR^4$ with smooth boundary such that $\delta \mapsto c_k(V^j_{\delta})$ is constant for all $k \neq j$, and $ c_j(V^j_{\delta}) = c_j(V^j_0)+\delta$.
\end{theorem}

\bigskip

\noindent{\bf If $U \subset \RR^{2n}$ is a star-shaped domain with smooth boundary, do the capacities $c_k(U)$   together with the volume of $U$ determine it up to symplectomorphism?} 

\medskip

This is a version of the {\em Recognition Question} from \cite{chls} (see  Question 2  in \S 3.6 of \cite{chls}).  Without the assumption that the boundary is smooth, the answer is known to be negative. For example, it follows from \cite{gh} that $c_k(P(1,2,6)) =c_k(P(1,3,4))=k$ for all $k \in \NN$. These polydisks also have the same volume. However, the rigidity of closed symplectic polydisks, established by L. Bates in \cite{bates} \footnote{It is still not known whether this rigidity holds for open polydisks.}, implies that they are not symplectomorphic.

Here we prove that the answer is still negative even under the additional assumption of smooth boundaries. Our main theorem is the following.
\begin{theorem}
\label{blind}
There is a smooth family $V_{\delta}$ of toric star-shaped domains in $\RR^4$ with smooth boundary all of which have the same Gutt-Hutching capacities and volume, but no two of which are symplectomorphic.
\end{theorem}

These domains are distinguished using the ECH capacities constructed by Hutchings in \cite{ech},  using the formulas for them established in \cite{choi}.

\subsection{Methods and Ideas}
The examples underlying the proofs of  Theorems \ref{novolume} , \ref{mutual} and \ref{blind} are all either convex or concave toric domains as defined by Gutt and Hutchings in \cite{gh}.  The key  tool developed here is a significant simplification of the combinatorial formulas  from \cite{gh} for the capacities of such regions, which holds in the presence of addition symmetry. As described below, these simplifications reveal new  blind spots of the capacities which are exploited to construct the relevant examples. They also reveal some new representation patterns for the capacities, and allow for new explicit computations.

\medskip

\noindent {\bf Symmetry and collapse.} We first recall the formulas for the capacities of convex and concave toric domains from \cite{gh}. Let $\mu \colon \CC^n=\RR^{2n} \to \RR_{\geq 0}^n$ be the standard moment map defined by
$$\mu(z_1,\dots, z_n) = (\pi|z_1|^2,\dots, \pi |z_n|^2).$$ To a domain
$\Omega$ in $\RR_{\geq 0}^n$ we associate the toric domain $X_{\Omega} = \mu^{-1}(\Omega)$ in $\RR^{2n}$. Following, \cite{gh}, the domain  $X_{\Omega}$ is said to be {\em convex} if 
$$
\hat{\Omega} = \{(x_1, \dots, x_n) \in \RR^n \mid (|x_1|, \dots, |x_n|) \in \Omega\}
$$
is compact and convex. The domain $X_{\Omega}$ is said to be {\em concave} if
$\Omega$ is compact and its complement, in $\RR_{\geq 0}^n$, is convex.

%Let $g \colon [0,A] \to \RR$ be a piecewise smooth, nonnegative and  nonincreasing function such that $g(A)=0$. Set
%\begin{equation}
%\label{ }
%\Omega_g= \left\{(x_1,x_2) \in \RR^2 \mid  0 \leq x_1 \leq A,\, 0 \leq x_2 \leq g(x_1)\right\}
%\end{equation}
%Given the standard moment map $\mu \colon \CC^2 \to \RR^2$ defined by $\mu(z_1,z_2) = (\pi|z_1|^2, \pi |z_2|^2)$, we define the toric domain corresponding to $g$ to be  
%\begin{equation}
%\label{ }
%X_g = \mu^{-1}(\Omega_g).
%\end{equation}
%If $g$ is concave the toric domain $X_g$ is said to be  convex and if $g$ is convex then $X_g$ is concave.
%

\begin{theorem}[\cite{gh}, Theorem 1.6 and Theorem 1.14]
\label{many}
If  $X_{\Omega} \subset \RR^{2n}$ is convex, then for any $k$ in $\NN$ we have 
\begin{equation*}
\label{ }
c_k (X_{\Omega}) = \min \setc*{ \|(v_1,\dots  v_n)\|_{\Omega}}{ v_1, \dots, v_n \in \{0\} \cup \mathbb{N}, \, \sum_{i=1}^n v_i=k}
\end{equation*}
where
\begin{equation*}
\label{ }
 \|(v_1,\dots  v_n)\|_{\Omega} =\max\left\{\langle v,w\rangle \mid w \in \Omega\right\}.
\end{equation*}

If  $X_{\Omega} \subset \RR^{2n}$ is concave,  then for any $k$ in $\NN$ we have 
\begin{equation*}
\label{ }
c_k (X_{\Omega}) = \max \setc*{ [(v_1,\dots  v_n)]_{\Omega}}{ v_1, \dots, v_n \in  \mathbb{N}, \, \sum_{i=1}^n v_i=k+n-1}
\end{equation*}
where
\begin{equation*}
\label{ }
[(v_1,\dots  v_n)]_{\Omega} =\min\left\{\langle v,w\rangle \mid w \in \bar{\partial}_+ \Omega\right\} 
\end{equation*}
and $\bar{\partial}_+\Omega$ is the closure of $\{w \in \partial \Omega \mid w \in \RR^n_{>0}\}$.
\end{theorem}

%\begin{remark} Given a domain $\Omega \subset \RR^n_{\geq 0}$ with topological boundary $\partial \Omega$,  let $\bar{\partial}_+ \Omega$ be the closure of $\{w \in \partial \Omega \mid w \in \RR^n_{>0}\}.$ It follows from the definitions above that in the formulas for $c_k(X_{\Omega})$ one can replace  $\|(v_1,\dots  v_n)\|_{\Omega}$ with $\max\left\{\langle v,w\rangle \mid w \in \bar{\partial}_+ \Omega\right\}$ and 
%$[(v_1,\dots  v_n)]_{\Omega}$ with $\min\left\{\langle v,w\rangle \mid w \in \bar{\partial}_+ \Omega\right\}$.
%\end{remark}

\begin{remark} In principle, to compute $c_k(X_{\Omega})$ with the formulae of Theorem \ref{many}, one must analyze ${k+n-1}\choose{n-1}$ optimization problems on $\Omega$ in the convex case, and ${k+n-2}\choose{n-1}$ such problems in the concave case.\footnote{These are the number of weak compositions of $k$, and the number of compositions of $k+n-1$, respectively.} 
\end{remark}

We say that  $\Omega \subset \RR^n_{\geq 0}$ is {\em symmetric} if
$$
(x_1, \dots, x_n) \in \Omega  \implies (x_{\sigma(1)}, \dots, x_{\sigma(n)}) \in \Omega
$$
for any permutation $\sigma \in S_n$. The following  result asserts that for symmetric convex domains $\Omega$ the formula for $c_k(X_{\Omega})$ from \cite{gh} collapses to a single optimization problem on $\bar{\partial}_+ \Omega$.

\begin{theorem}
\label{balance}
If $\Omega \subset \RR^n_{\geq 0}$ is symmetric and $X_{\Omega}$ is convex, then 
\begin{equation*}
\label{ }
c_k(X_{\Omega}) =\max\left\{\langle V(k,n),w\rangle \mid w \in \bar{\partial}_+ \Omega\right\}.
\end{equation*}
where 
\begin{equation*}
\label{ }
V(k,n) = \left( \left\lfloor \frac{k}{n} \right\rfloor, \dots,  \left\lfloor \frac{k}{n}\right \rfloor,  \underbrace{\left\lceil \frac{k}{n}\right\rceil , \dots, \left\lceil \frac{k}{n}\right\rceil}_{(k \mod n)-\mathrm{times}} \right).
\end{equation*}
\end{theorem}

A similar collapse occurs for symmetric concave domains. 

\begin{theorem}
\label{balance-c}
If $\Omega \subset \RR^n_{\geq 0}$ is symmetric and $X_{\Omega}$ is concave, then 
\begin{equation*}
\label{ }
c_k(X_{\Omega}) =\min\left\{\langle \check{V}(k,n),w\rangle \mid w \in \bar{\partial}_+ \Omega\right\}
\end{equation*}
where 
\begin{equation*}
\label{ }
\check{V}(k,n) = \left(   \underbrace{\left\lceil \frac{k+n-1}{n}\right\rceil , \dots, \left\lceil \frac{k+n-1}{n}\right\rceil}_{(k +n-1\mod n)-\mathrm{times}},\left\lfloor \frac{k+n-1}{n} \right\rfloor, \dots,  \left\lfloor \frac{k+n-1}{n}\right \rfloor\right).
\end{equation*}
\end{theorem}

\medskip

\noindent {\bf New blind spots.}
The following simple case demonstrates how Propositions \ref{balance} and \ref{balance-c} can be used to reveal new and useful blind spots of the Gutt-Hutching's capacities. Let $f \colon [0,1] \to \RR_{\geq 0}$ be a  continuous  strictly decreasing  function such that $f(0)=1$ and $f(1)=0$. Let  $\Omega_f$ be the domain in $\RR^2_{\geq 0}$ that is bounded by the axes and the graph of $f$ and  set $X_f= \mu^{-1}(\Omega_f)$.  Suppose further that $f$ is smooth on $[0,1)$ and satisfies $f'(0)=0$, $f''<0$ and $f^{-1} = f$.  Then $X_f$ is a (strictly) convex toric domain in $\RR^4$ which is symmetric. In this case, Theorem  \ref{balance} implies that 
\begin{equation*}
\label{ }
c_k(X_{\Omega}) =\max\left\{\langle V(k,2),(x,f(x))\rangle \mid x \in [0,1]\right\}.
\end{equation*}
A simple argument, see Proposition \ref{sym} and Corollary \ref{f} below, implies 
that for even values of $k$ this reduces to 
\begin{equation*}
c_k (X_f) = k x(f),
\end{equation*}
where $x(f)$ is the unique fixed point of $f$. For odd values of $k$ it reduces to
 \begin{equation*}
c_k (X_f) = \frac{k-1}{2}x_k + \frac{k+1}{2}f(x_k)
\end{equation*}
where $x_k$ is the unique solution of
\begin{equation}
\label{hit}
f'= -\frac{k-1}{k+1}.
\end{equation}
{\bf The important point here is that the capacities $c_k(X_f)$ are all determined by the $1-$jet of $f$ at $x(f)$ and at the sequence of points $x_k$, for $k$ odd, that converges monotonically to $x(f)$.}   This is illustrated in Example \ref{round} and Figure \ref{red}, below. 

The {\em blind spots} referred to in the title of this paper are epitomized by the values of $f$ away from the points of the  sequence $x_k \to x(f)$ and, by symmetry, its mirror sequence $f(x_k) \to x(f)$. By altering $f$ away from all these points, in a manner which preserves its crucial properties, one can change the volume while keeping the capacities fixed. By altering $f$ near one of the points  $x_j$ and its mirror image $f(x_j)$, in a manner which preserves the function's crucial properties, one can vary the $j$th capacity while leaving the others unchanged. One can also alter $f$ away from all these points, in a manner which preserves its crucial properties,  and deform the region, while keeping its volume fixed, with the goal  of changing its symplectomorphism class. These  are the simple ideas that underly the proofs of Theorems \ref{novolume}, \ref{mutual}, and \ref{blind}, respectively. \footnote{These ideas were motivated by the examples that appear in the authors' previous work, \cite{kl}.} 

\medskip

\begin{example}
\label{round}
For the function $f_2(x)=\sqrt{1-x^2}$, the capacities $c_k(X_{f_2})$ are determined by the values of $f_2$ on the sequence of points 
$$x_k = \frac{k-1}{\sqrt{2k^2+2}}, \text{ for odd $k$}$$
and its limit $x(f_2) =\frac{1}{\sqrt{2}}$. They are given by the formula
\begin{equation*}
\label{ }
c_k (X_{f_2}) = 
\begin{cases}
    \frac{k}{\sqrt{2}}  &, \text{ if $k$ is even}, \\
    {} & {}\\
   \frac{\sqrt{k^2+1}}{\sqrt{2}}    &, \text{ if $k$ is odd}.
\end{cases}
\end{equation*}
\end{example}

\begin{figure}[H]
\centering
\caption{A view of the infinite symplectic rib cage of $X_{f_2}$.}
\bigskip
  \includegraphics[width=10cm]{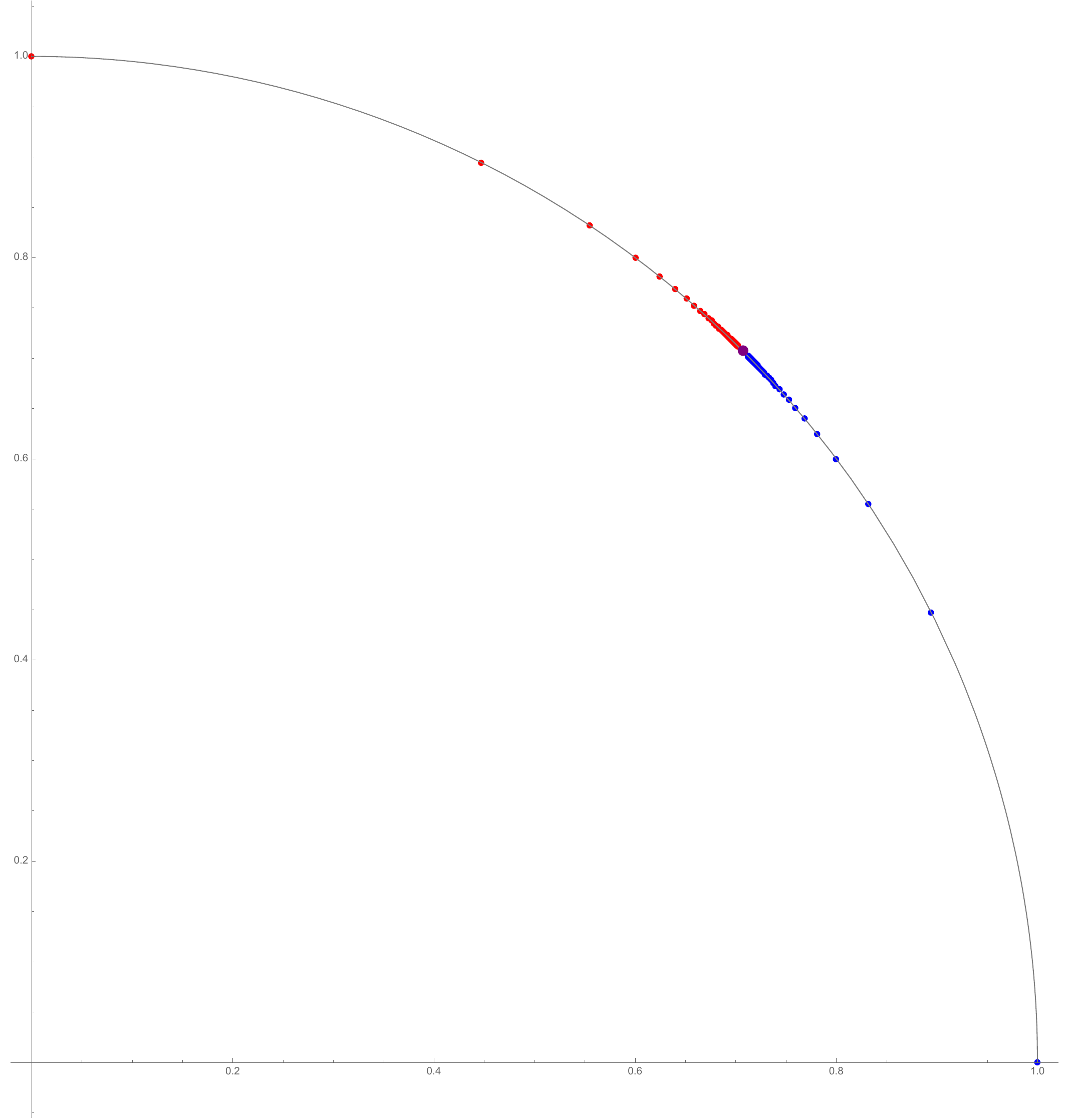}
  \label{red}
  
  \medskip
The red points correspond to the sequence $(x_{2n-1}, f_2(x_{2n-1}))$ and determine all the odd index capacities of $X_{f_2}$. They converge to the purple point $(x(f_2), x(f_2))$ which determines all the even index capacities of $X_{f_2}$.
\end{figure}

%\begin{figure}[!h]
%\centering
%\caption{The graph of  $p \mapsto c_{123}(E_p(1,e))$.}
%\bigskip
%  \includegraphics[width=9cm]{Fig2.pdf}
%  \label{e2p}
%\end{figure}

\medskip

\medskip

\noindent {\bf Infinite symplectic rib cages.}  The domains $X_f$ above yield the first examples of star-shaped domains whose capacities are represented by  infinitely many geometrically distinct  closed characteristics. In the spirit of Arnold's evocative analogy from \cite{ar-first}, we refer to the collection of these characteristics as an {\em infinite symplectic rib cage}. 

The carriers of the odd index capacities are distinct ribs, that is distinct closed characteristics. A representative for $c_1(X_f)$ is a closed characteristic with image $\mu^{-1}((0,1))$.  For each odd $k=2n-1>1$ the capacity $c_{2n-1}(X_f)$ is represented by (any of) the closed characteristics on the Lagrangian torus $\mu^{-1}((x_k, f(x_k))) \subset \partial X_f$. These characteristics represent the class $(n-1,n)\in \mathrm{H}_1(S^1 \times S^1, \ZZ) \simeq \ZZ^2$ with respect to the standard basis. 

For different even values, $k=2n$, the capacity carriers may not be geometrically distinct. In particular, the capacity $c_{2n}(X_f)$ is represented by (any of) the closed characteristics on the Lagrangian torus $\mu^{-1}((x(f),x(f))) \subset \partial X_f$ in class $(n,n)$. So, for example, representatives of $c_4(X_f)$ cover representatives of $c_2(X_f)$. It is a curious fact that this covering phenomenon corresponds to the monotone Lagrangian torus $\mu^{-1}((x(f),x(f)))$.

One also observes infinite symplectic rib cages for concave toric domains of the form $X_h$ where $h$ is convex ($h''\geq 0$), see Proposition \ref{sym-g}. In that case, all the even index capacities are represented by distinct closed characteristics.

\begin{question}\label{allrib}
Is there a star-shaped domain $U$ of $\RR^{2n}$  such that {\bf all} the capacities $c_k(U)$ are represented by  distinct closed characteristics?
\end{question}

\medskip

\noindent {\bf New capacity computations.}
Propositions \ref{balance} and \ref{balance-c}  yield formulas for the capacities of many interesting domains. Several of these computations are presented in Section \ref{examples}. Here we mention just the following.
\begin{example}[The Lagrangian Bidisk]
\label{lag}
Let $\Omega_L$ be the concave region in $\RR^2_{\geq 0}$ bounded by the axes and the curve 
\begin{equation*}
\label{ }
\alpha(t) =\left( 2 \sin \left(\frac{t}{2}\right)  -t \cos \left(\frac{t}{2}\right),   2 \sin \left(\frac{t}{2}\right)  + (2\pi -t) \cos \left(\frac{t}{2}\right)\right),\quad t \in [0,2\pi].
\end{equation*}
In \cite{ra}, Ramos proves that the interior of the corresponding toric domain $X_{\Omega_L}$ is symplectomorphic to the 
interior of the {\em Lagrangian bidisk}
$$
P_L= \left\{ (x_1+i y_1, x_2+ iy_2) \in \CC^2 \mid  x_1^2 +x_2^2 \leq 1, \, y_1^2 +y_2^2 \leq 1\right\}.
$$
Theorem  \ref{balance-c} implies that for an odd $k$ we have  
\begin{eqnarray*}
c_k (X_{\Omega_L})& = & \min_{t \in [0,2\pi]} \frac{1}{2}\left\{\langle (k+1, k+1), \alpha(t) \rangle\right\}\\
 & = & 2k+2. 
\end{eqnarray*}
For even $k$ we have  
\begin{eqnarray*}
c_k (X_{\Omega_L})& = & \min_{t \in [0,2\pi]} \frac{1}{2}\left\{\langle (k+2, k), \alpha(t) \rangle\right\}\\ & = &  (2k+2) \sin \left( \frac{\pi}{2}\left(\frac{k}{k+1}\right)\right)
\end{eqnarray*}
Invoking  Ramos's symplectomorphism from \cite{ra}, we then get the following simple formula for the Gutt-Hutchings capacities of the Lagrangian Bidisk:
\begin{equation}
\label{super}
c_k (P_L) = 
\begin{cases}
2k+2,  & \text{for odd $k$}\\ \\
 (2k+2) \sin \left( \frac{\pi}{2}\left(\frac{k}{k+1}\right)\right), & \text{for even  $k$}. \\
\end{cases}
\end{equation}
\end{example}

\begin{remark} The first three Ekeland-Hofer capacities of $P_L$ were computed only recently by Baracco, Fassina and  Pinton in \cite{bfp}.
\end{remark}

\medskip

\noindent {\bf Rigidity of symplectic embeddings to and from balls.}  The following is a simple application of Propositions \ref{balance} and \ref{balance-c}.

\begin{corollary}
\label{ball}
If $X_{\Omega} \subset \RR^{2n}$ is a symmetric convex toric domain, then 
\begin{equation}
\label{c1}
c_{1} (X_{\Omega}) = \max \{\delta \mid (0, \dots,0, \delta) \in \Omega\},
\end{equation}
and for every $\ell \in \NN$ we have 
\begin{equation}
\label{cln}
c_{\ell n} (X_{\Omega}) = \ell n \max \{\delta \mid (\delta, \dots, \delta) \in \Omega\}.
\end{equation}
If $X_{\Omega} \subset \RR^{2n}$ is a symmetric concave toric domain, then for every $\ell \in \NN$ we have 
\begin{equation}
\label{vln}
c_{n(\ell-1)+1} (X_{\Omega}) = \ell n \max \{\delta \mid (\delta, \dots, \delta) \in \Omega\}.
\end{equation}
\end{corollary}

\begin{proof}
For the convex case, Theorem  \ref{balance} implies that for $k=1$ we have 
\begin{equation*}
\label{ }
c_{1} (X_{\Omega}) = \max\{ w_n \mid (w_1, \dots w_n) \in \bar{ \partial}_+ \Omega\}
\end{equation*}
which is equivalent to \eqref{c1}. For $k=\ell n$, Theorem  \ref{balance-c} implies 
\begin{equation*}
\label{ }
c_{\ell n} (X_{\Omega}) = \max\{ \ell(w_1 + \dots + w_n) \mid (w_1, \dots w_n) \in \bar{ \partial}_+ \Omega\}
\end{equation*}
By continuity of the capacities we may assume that $\bar{\partial}_+\Omega$ is smooth. It then follows from the symmetry assumption  and elementary constrained optimization theory, that $\max\{ \ell(w_1 + \dots + w_n) \mid (w_1, \dots w_n) \in \bar{ \partial}_+ \Omega\}$ is realized at the intersection of $\bar{\partial}_+ \Omega$ with the super diagonal.  Equation \eqref{cln}, follows from this. 

For the concave case, Theorem  \ref{balance-c} implies that for $k=n(\ell-1)+1$
\begin{equation*}
\label{ }
c_{n(\ell-1)+1} (X_{\Omega}) = \max\{ \ell(w_1 + \dots + w_n) \mid (w_1, \dots w_n) \in \bar{ \partial}_+ \Omega\}.
\end{equation*}
A similar argument to that above yields \eqref{vln}.

\end{proof}

Corollary \ref{ball} immediately implies a rigidity result. For a symmetric convex toric domain $X_{\Omega} \subset \RR^{2n}$, the largest standard ball contained in $X_{\Omega}$ has radius equal to $$\sqrt{\max \{\delta \mid (0, \dots,0, \delta) \in \Omega\}/\pi}$$  and the largest standard ball containing  $X_{\Omega}$ has radius equal to $$\sqrt{n \max \{\delta \mid (\delta, \dots, \delta) \in \Omega\}/\pi}.$$ Applying Corollary \ref{ball}, it follows that one can not do better with nonstandard balls. In particular, we have the following.
\begin{theorem}\label{balls}
If $X_{\Omega} \subset \RR^{2n}$ is a symmetric convex toric domain, then 
\begin{equation}
\label{sup}
\sup\{a \mid \text{  $\exists$ a symplectic embedding  } B^{2n}(a) \to X_{\Omega}\} = \max \{\delta \mid (0, \dots,0, \delta) \in \Omega\}
\end{equation}
and 
\begin{equation}
\label{inf}
\inf \{A \mid \text{   $\exists$ a symplectic embedding  } X_{\Omega}  \to B^{2n}(A) \} = n \max \{\delta \mid (\delta, \dots, \delta) \in \Omega\}.
\end{equation}
\end{theorem}

Equation \eqref{sup} also holds if $X_{\Omega}$ is concave. For concave  toric domains, equation \eqref{inf} is known to be false in general (see, for example, \cite{or}).

\subsection{Organization}
Section \ref{forgetting} contains an analysis of how the capacities $c_k(E_p(1,a))$ loose sight of $a$ as $p$ goes to infinity. The simplified capacity formulas of Theorem \ref{balance} and Theorem  \ref{balance-c} are established in Section \ref{collapse}. These formulas are then used,  in Section \ref{examples}, to derive explicit formulas for the Gutt-Hutchings capacities in several new examples. The proofs of Theorems \ref{novolume},\ref{mutual} and \ref{blind} are presented in Section \ref{proofs}. The last section contains a discussion of some questions motivated by these results and their proofs.

\section{The development of blind spots at corners}\label{forgetting} For a fixed $a>1$, consider the family of domains 
\begin{equation}
\label{pellipse}
E_p(1,a) = \left\{ (z_1, z_2) \in \CC^2 \mid  (\pi |z_1|^2)^p + \left(\frac{\pi |z_2|^2}{a}\right)^p \leq 1\right\}.
\end{equation}
As $p$ varies from  $1$ to $\infty$ these domains connect the ellipse $E(1,a)$ to the polydisk $P(1,a)$. In this section we analyze the process by which the capacities $c_k(E_p((1,a))$ loose sight of $a$ along the way.

Each $E_p(1,a)$ is a convex toric domain of the form  $\mu^{-1}({\Omega_{f_p}})$ where $\Omega_{f_p}$ is the region in $\RR^2_{\geq0}$ bounded by the axes and the graph of the smooth strictly concave function  $$f_p(x) = a (1-x^p)^{\frac{1}{p}}.$$ 
Since $a>1$, this domain is not symmetric. However, the formula from Theorem \ref{many}  can still be significantly simplified in this setting. 

Working a little more generally, let  $\mathcal{V}$ be the set of  continuous functions $f \colon [0,1] \to \RR_{\geq 0}$  with the following properties:
\begin{enumerate}
 \item[(1)]  $f(0)\geq1$.
 \item[(2)]  $f(1)=0$.
 \item[(3)] $f$ is smooth on $[0,1)$.
 %\item $f'(0) > -\frac{1}{2}$
  \item[(4)] $f'(0)=0$.
  \item[(5)] $\lim_{x\to 1^-}f'(x)=-\infty$.
  \item[(6)] $f'' <0$ on $(0,1)$. 
\end{enumerate}

Each $f \in  \mathcal{V}$ has  a unique fixed point in $(0,1)$ which we denote by $x(f)$. For $k \geq 2$ and $j=0, \dots k-2$ set 
$$
I^k_j =\left(\frac{-j-1}{k-j-1}, \frac{-j}{k-j}\right].
$$
Let
$
I^k_{k-1} =\left(-\infty, -k+1\right].
$
We then define $J_k =J_k(f)$ to be the integer in $[0, k-1]$ that is determined uniquely by the condition
\begin{equation*}
\label{ }
f'(x(f)) \in I^k_{J_k}.
\end{equation*}

\begin{proposition}\label{messy}
 For $f \in \mathcal{V}$, let $\Omega_f$ be the region in $\RR^2_{\geq 0}$ bounded by the axes and the graph of $f$. Let $X_f= \mu^{-1}(\Omega_f)$. Then $c_1(X_f)=1$ and for $k \geq 2$ 
\begin{equation}
\label{form}
c_k(X_f) = \min \{k, F(k,f)\} 
\end{equation}
where 
 \begin{equation*}
\label{ }
F(k,f)=\begin{cases}
     x_{1} +(k- 1) f(x_{1})   & \text{if } J_k=0, \\
   \min\left\{ J_k x_{J_k} + (k-J_k)f(x_{J_k}), (J_k+1)x_{J_k+1} + (k-J_k-1)f(x_{J_k+1}) \right\}   & \text{if }  1 \leq J_k < k-1\\
(k-1) x_{k-1} + f(x_{k-1})   & \text{if }  J_k = k-1,
\end{cases}
\end{equation*}
and $x_{\ell} =x_{\ell}(f)$ is the unique solution of $f'(x) = \frac{-\ell}{k-\ell}$.
\end{proposition}

%\begin{remark}
%For each $k$, the computation of $c_k(X_f)$ using Proposition \ref{messy} involves the computation of four new terms; $J_k$, $x_{J_k}$, $x_{J_k+1}$, and $x_{k-1}$. In comparison, the original formula of Theorem \ref{many} requires the analysis of $k+1$ optimization problems.
%\end{remark}

\begin{proof}
Since the compact convex domain $\Omega_f$ is defined by the graph of $f$, the formula for 
$c_k(X_f)$ can be simplified to 
\begin{equation*}
\label{ }
c_k(X_f) = \min_{\ell \in \{0,1,\dots,k\}}\left\{\max_{x \in [0,1]}\{\ell x+ (k-\ell)f(x)\}\right\}.
\end{equation*}
The assertion about $c_1(X_f)$ follows immediately. Assuming now that $k\geq2$, consider the partition
\begin{equation*}
\label{ }
\{0,1, \dots, k\}= \underbrace{\{0,k\}}_{A_k} \cup \underbrace{\{1, \dots, k-1\}}_{B_k}.
\end{equation*}
We then have
$$
\min_{\ell \in A_k}\left\{\max_{x \in [0,1]}\{\ell x+ (k-\ell)f(x)\}\right\}=\min\{ kf(0), k\} =k.
$$ 

 The condition that $f \in \mathcal{V} $ satifsies  $f''<0$ on $(0,1)$, implies that for $\ell$ in $B_k$ the function $$x \mapsto \ell x+ (k-\ell)f(x)$$ has a unique critical point at the point $x_{\ell}$ defined uniquely by 
the equation $f'(x) = \frac{-\ell}{k-\ell}$. Moreover,  $x_{\ell}$ is a global maximum. Hence, for $\ell \in B_k$ we have
\begin{equation*}
\label{ }
\max_{x \in [0,1]}\{\ell x+ (k-\ell)f(x)\} = \ell x_{\ell} + (k-\ell)f(x_{\ell}).
\end{equation*}

Now consider
$$
c(\ell)  = \ell x_{\ell} + (k-\ell)f(x_{\ell})
$$
as a function of a {\bf real} variable $\ell \in [1,k-1]$ where $x_{\ell}$ is now the smooth function defined implicitly by the equation $f'(x_{\ell})=\frac{-\ell}{k-\ell}$. We then have
\begin{equation}
\label{ell}
c'(\ell) = x_{\ell} - f(x_{\ell}).
\end{equation}
Hence, $c(\ell)$ takes its minimum value at the value of $\ell$ corresponding to the fixed point $x(f)$. 

Viewing $\ell$ as a {\bf discrete} variable in $\{1,2, \dots, k-1\}$,  the discussion above implies that  if $J_k$ is contained in $[1,k-1)$, then 
\begin{equation*}
\label{ }
\min_{\ell \in B_k}\left\{\max_{x \in [0,1]}\{\ell x+ (k-\ell)f(x)\}\right\} =  \min\left\{ J_k x_{J_k} + (k-J_k)f(x_{J_k}), (J_k+1)x_{J_k+1} + (k-J_k-1)f(x_{J_k+1}) \right\}.
\end{equation*}
On the other hand, if $J_k=0$, then \eqref{ell} implies that $c(\ell)$ is increasing (as a function on the discrete set $\{1,2, \dots, k-1\}$) and so 
\begin{equation*}
\label{ }
\min_{\ell \in B_k}\left\{\max_{x \in [0,1]}\{\ell x+ (k-\ell)f(x)\}\right\} =  x_{ 1} +(k- 1)f(x_1).
\end{equation*}
Similarly, if $J_k\geq k-1$, then \eqref{ell} implies that $c(\ell)$ is decreasing  and 
\begin{equation*}
\label{ }
\min_{\ell \in B_k}\left\{\max_{x \in [0,1]}\{\ell x+ (k-\ell)f(x)\}\right\} = (k-1) x_{k-1} +f(x_{k-1}).
\end{equation*}
Putting these observations together we arrive at the formula of Proposition \ref{messy}.
\end{proof}

We now apply Proposition \ref{messy} to the function $f_p(x) = a (1-x^p)^{\frac{1}{p}}$ defining  $E_p(1,a)$. We first  prove that for each $k$ the  capacity $c_k(E_p(1,a))$ looses sight of $a$ for all sufficiently large $p$.

\begin{lemma}\label{k}
For each $k \in \NN$ there is a $p(k) \in [1, \infty)$ such that $c_k(E_p(1,a)) =k$  for all $p\geq p(k)$.
\end{lemma}

\begin{proof} A simple computation  yields
\begin{equation*}
\label{ }
f'_p(x(f_p)) =-a^p.
\end{equation*}
For all $\ell \in [1,k-1]$ we also have
\begin{equation*}
\label{ }
x_{\ell}(p) = \frac{\left(\frac{\ell}{(k-\ell)a}\right)^{\frac{1}{p-1}}}{\left(1+\left(\frac{\ell}{(k-\ell)a}\right)^{\frac{p}{p-1}}\right)^{\frac{1}{p}}}
\end{equation*}
and
\begin{equation*}
\label{ }
f_p(x_{\ell}(p))=a\left(1+\left(\frac{\ell}{(k-\ell)a}\right)^{\frac{p}{p-1}}\right)^{-\frac{1}{p}}.
\end{equation*}

Consider a fixed $k \geq 2$. For $p$ near $1$, the capacity $c_k(E_p(1,a))$ is determined by which interval, say $I^k_{j_1}$, contains $-a$. If $j_1=k-1$, then $-a^p$ remains in $I^k_{k-1}$ as $p$ increases. Otherwise, $-a^p$ moves from $I^k_{j_1}$ to  $I^k_{j_1+1}$ and continues in this way until it enters  $I^k_{k-1}$ for good. In either case, for $p$ sufficiently large we  have 
\begin{eqnarray*}
F(k,f_p) & = & (k-1)x_{k-1}+ f_p(x_{k-1})\\
%{} & = & \left(1+\left(\frac{k-1}{a}\right)^{\frac{p}{p-1}}\right)^{-\frac{1}{p}}\left((k-1)\left(\frac{k-1}{a}\right)^{\frac{1}{p-1}}+ a\right)\\
%{} & = & a\left(1+\left(\frac{k-1}{a}\right)^{\frac{p}{p-1}}\right)^{\frac{p-1}{p}}\\
{} & = & \left(a^{\frac{p}{p-1}}+\left(k-1\right)^{\frac{p}{p-1}}\right)^{\frac{p-1}{p}}.
\end{eqnarray*} 
Lemma \ref{k}, now follows from Proposition \ref{messy} (see \eqref{form}) and the fact that
\begin{equation*}
\label{ }
\lim_{p\to \infty}F(k,f_p) =k-1+a>k.
\end{equation*}

\end{proof}

Figure \ref{e2p} illustrates the formula from Proposition \ref{messy} and the behavior described in  Lemma \ref{k} when $k=123$ and $a$ is equal to Euler's number, $e$.
\begin{figure}[!h]
\centering
\caption{The graph of  $p \mapsto c_{123}(E_p(1,e))$.}
\bigskip
  \includegraphics[width=9cm]{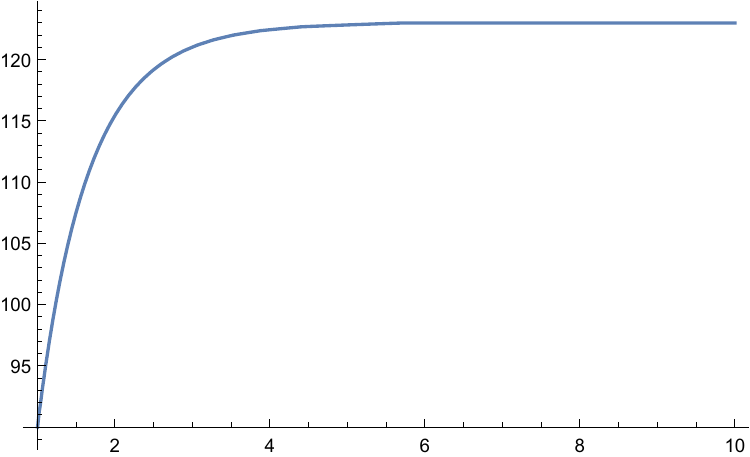}
  \label{e2p}
\end{figure}

Next we prove that  $a$ is visible to the collection of capacities, $\{c_k(E_p(1,a))\}_{k \in \NN}$ for all $p<\infty$.

\begin{lemma}\label{p}
 For each $p \in (1, \infty)$ there is a $k(p) \in \NN$ such that 
 $c_k(E_p(1,a))$  depends on $a$ for all $k\geq k(p)$.
\end{lemma}

\begin{proof} Suppose that
\begin{equation}
\label{in}
a^p<k-1.
\end{equation} 
Then  $-a^p$ lies in some some $I^k_j$ for $j \leq k-2$ and  
\begin{equation*}
\label{ }
F(k,f_p) =\left((a(k-m))^{\frac{p}{p-1}}+m^{\frac{p}{p-1}}\right)^{\frac{p-1}{p}}
\end{equation*}
for some $m\geq 1$. It follows from the proof of Proposition \ref{messy} that 
\begin{eqnarray*}
F(k,f_p) &\leq & (k-1)x_{k-1} + f_p(x_{k-1}) \\
{} & = & \left(a^{\frac{p}{p-1}}+(k-1)^{\frac{p}{p-1}}\right)^{\frac{p-1}{p}}
\end{eqnarray*}
Hence, one has  $F(k,f_p) <k$ if 
\begin{equation*}
\label{under}
a^p < \left(k^{\frac{p}{p-1}} -(k-1)^{\frac{p}{p-1}}\right)^{p-1}.
\end{equation*}
The right hand side goes to infinity as $k$ does. 
%
%need $\left(k^{\frac{p}{p-1}} -(k-1)^{\frac{p}{p-1}}\right) \to \infty$
%
%$k^{\frac{p}{p-1}}\left(1 -(\frac{k-1}{k})^{\frac{p}{p-1}}\right) \to \infty$
%
%$\lim \frac{1 -(\frac{k-1}{k})^{\frac{p}{p-1}}}{k^{\frac{-p}{p-1}}}$
%
%$\lim \frac{-\frac{p}{p-1}(\frac{k-1}{k})^{\frac{1}{p-1}}\frac{1}{k^2}}{-\frac{p}{p-1}k^{\frac{-2p+1}{p-1}}}$
%
%$\lim (k-1)^{\frac{1}{p-1}}$
%
With this, Proposition \ref{messy} implies that if  $k$ is large enough for inequalities \eqref{in} and \eqref{under} to both hold, then  
\begin{equation*}
\label{ }
c_k(E_p(1,a)) = \left((a(k-m))^{\frac{p}{p-1}}+m^{\frac{p}{p-1}}\right)^{\frac{p-1}{p}}
\end{equation*}
for some $1\leq m\leq k-1$.

\end{proof}

\section{Symmetry and collapse}\label{collapse} Here we present the proofs of Theorem
\ref{balance} and Theorem  \ref{balance-c}. Recall that  $\Omega \subset \RR^n_{\geq 0}$ is {\em symmetric} if
$$
(x_1, \dots, x_n) \in \Omega  \implies (x_{\sigma(1)}, \dots, x_{\sigma(n)}) \in \Omega
$$
for any permutation $\sigma \in S_n$.
\subsection{Proof of Theorem
\ref{balance}}
We must show that if $\Omega \subset \RR^n_{\geq 0}$ is symmetric and $X_{\Omega}$ is convex, then 
\begin{equation*}
\label{ }
c_k(X_{\Omega}) = \|V(k,n)\|_{\Omega}
\end{equation*}
where 
\begin{equation*}
\label{ }
V(k,n) = \left( \left\lfloor \frac{k}{n} \right\rfloor, \dots,  \left\lfloor \frac{k}{n}\right \rfloor,  \underbrace{\left\lceil \frac{k}{n}\right\rceil , \dots, \left\lceil \frac{k}{n}\right\rceil}_{k\mod n} \right).
\end{equation*}

We say that an $n$-tuple $y=(y_1, \dots y_n) \in \RR^n$ is {\em ordered} if  
$$y_1 \leq y_2 \leq \dots \leq y_n.$$
Since $\Omega$ is symmetric,  we have 
\begin{equation}
\label{inv}
\|(v_1, \dots v_n)\|_{\Omega} = \|(v_{\sigma(1)}, \dots v_{\sigma(n)})\|_{\Omega}
\end{equation}
for all $\sigma \in S_n$. This allows us to rewrite the expression for $c_k (X_{\Omega})$, as 
\begin{equation*}
\label{ }
c_k (X_{\Omega}) = \min \setc*{ \|v\|_{\Omega}}{ v \in \vec{C}(k,n)}
\end{equation*}
where
$$\vec{C}(k,n) = \setc*{ v \in \ZZ^n_{\geq 0}}{\sum_{i=1}^n v_i=k,\, v \text{ is  ordered}}.$$

\begin{lemma}\label{exist}
For every $v \in \vec{C}(k,n)$ there exists an ordered $w= (w_1, \dots, w_n) \in \Omega$ such that
$\|v\|_{\Omega} = \langle v,w\rangle.$
\end{lemma}
\begin{proof}
Choose $\tilde{w}=(\tilde{w}_1, \dots, \tilde{w}_n)$ in $\Omega$  such that $\|v\|_{\Omega} =\langle v,\tilde{w}\rangle.$ If $\tilde{w}$ is ordered we are done. Otherwise,  $\tilde{w}_i > \tilde{w}_j$ for some $1 \leq i < j \leq n$ . Let $\tau \in S_n$ be the transposition $\tau = (i \;  j)$, and set $\tilde{w}_\tau = (\tilde{w}_{\tau(1)}, \dots , \tilde{w}_{\tau(n)})$. Since $\Omega$ is symmetric $\tilde{w}_\tau$ is in $\Omega$ and  
\begin{equation}
\begin{split}
\langle v,\tilde{w}_{\tau}\rangle- \langle v,\tilde{w}\rangle
& = \ (v_{i}\tilde{w}_{j} + v_{j}\tilde{w}_{i})-(v_{i}\tilde{w}_{i} + v_{j}\tilde{w}_{j}) \\
&= \ (v_i - v_j)(\tilde{w}_j - \tilde{w}_i) \\
& \geq 0. 
\end{split}
\end{equation}
As $\langle v,\tilde{w}\rangle$ is already realizing the maximum, $\|v\|_{\Omega}$, this inequality must be an equality. Hence $\|v\|_{\Omega} = \langle v,\tilde{w}_{\tau}\rangle.$ Proceeding in this manner we can continue to order $\tilde{w}$ and obtain the desired
$w \in \Omega$.
\end{proof}

For $v \in \vec{C}(k,n)$, set
\begin{equation*}
\label{ }
\Delta(v) = v_n(\#\{v_i \mid v_i=v_n\}) - v_1(\#\{v_i \mid v_i=v_1\}).
\end{equation*}
Consider the {\em transfer} map $\mathcal{T} \colon \vec{C}(k,n) \to \vec{C}(k,n)$ defined by
$$v=(\underbrace{v_1, \dots v_1}_{t-times}, \dots , \underbrace{v_n, \dots v_n}_{T-times}) \mapsto  
\begin{cases}
      (\underbrace{v_1, \dots v_1}_{(t-1)-times}, v_1+1 \dots , v_n-1, \underbrace{v_n, \dots v_n}_{(T-1)-times}) &\text{ if $v_n >v_1+1$}, \\
      v &  \text{ otherwise}.
\end{cases}
$$
Note that the vector $V(k,n)$ in the statement of Theorem  \ref{balance} is the unique fixed point of $\mathcal{T}$ and that $\Delta(\mathcal{T}(v)) \leq \Delta(v)$ with equality if and only if $v=V(k,n)$. Hence, 
\begin{equation}
\label{fix}
\mathcal{T}^j(v)= V(k,n) \text{ for all sufficiently large } j.
\end{equation}

\begin{lemma}
\label{transfer}
For all $v \in \vec{C}(k,n)$ we have  $\|\mathcal{T}(v)\|_{\Omega}  \leq \|v\|_{\Omega} $.
\end{lemma}
\begin{proof}
By Lemma \ref{exist} there is an ordered $w \in \Omega$ such that  $\|\mathcal{T}(v)\|_{\Omega} = \langle \mathcal{T}(v),w \rangle$. We then have 
\begin{equation}
\begin{split}
\|v\|_{\Omega} - \|\mathcal{T}(v) \|_{\Omega} 
& \geq \ \langle v, w \rangle - \langle \mathcal{T}(v), w  \rangle \\
& = \ (v_{1}w_{t} + v_{n}w_{n-T}) - ((v_{1} + 1) w_{t} + (v_{n} - 1)w_{n-T}) \\
&= \ w_{n-T} - w_t\\
\end{split}
\end{equation}
which is nonnegative since $w$ is ordered.
\end{proof}

Since  $c_k(X_{\Omega}) = \|v\|_{\Omega}$  for some $v$ in $\vec{C}(k,n)$, Lemma \ref{transfer} implies that $$c_k(X_{\Omega}) = \|\mathcal{T}^j(v)\|_{\Omega}$$ for any $j \in \NN$.  It then follows from \eqref{fix} that  
$c_k(X_{\Omega}) = \|V(k,n)\|_{\Omega}$. This completes the proof of Theorem  \ref{balance}.

\subsection{Proof of Theorem \ref{balance-c}}
For a symmetric  $\Omega \subset \RR^n_{\geq 0}$ such that $X_{\Omega}$ is concave, we must show that  
\begin{equation*}
\label{ }
c_k(X_{\Omega}) = [\check{V}(k,n)]_{\Omega}
\end{equation*}
where 
\begin{equation*}
\label{ }
\check{V}(k,n) = \left(   \underbrace{\left\lceil \frac{k+n-1}{n}\right\rceil , \dots, \left\lceil \frac{k+n-1}{n}\right\rceil}_{k+n-1 \mod n},\left\lfloor \frac{k+n-1}{n} \right\rfloor, \dots,  \left\lfloor \frac{k+n-1}{n}\right \rfloor\right).
\end{equation*}

As the proof is very similar to that of Theorem  \ref{balance} it is presented in less detail. A vector $y=(y_1, \dots y_n) \in \RR^n$ is said to be  {\em ordered backwards} if  
$$y_1 \geq y_2 \geq \dots \geq y_n.$$ For symmetric concave domain $X_{\Omega}$ we then have 
\begin{equation*}
\label{ }
c_k (X_{\Omega}) = \max \setc*{ [v]_{\Omega}}{ v \in \cev{C}(k,n)},
\end{equation*}
where $\cev{C}(k,n) = \setc*{v \in \NN^n}{ \sum_{i=1}^n v_i=k+n-1,\, v \text{ is ordered backwards}}.$

\begin{lemma}
\label{real-c}
For every $v \in \cev{C}(k,n)$ there exists an ordered $w= (w_1, \dots, w_n) \in \Omega$ such that
$[v]_{\Omega} = \langle v,w\rangle.$
\end{lemma}
\begin{proof}
Choose $\tilde{w}=(\tilde{w}_1, \dots, \tilde{w}_n)$ in $\Omega$  such that $[v]_{\Omega} =\langle v,\tilde{w}\rangle.$ If $\tilde{w}$ is ordered we are done. Otherwise,  $\tilde{w}_i > \tilde{w}_j$ for some $1 \leq i < j \leq n$ . Let $\tau \in S_n$ be the transposition $\tau = (i \;  j)$, and set $\tilde{w}_\tau = (\tilde{w}_{\tau(1)}, \dots , \tilde{w}_{\tau(n)})$. We then have 
\begin{equation}
\begin{split}
\langle v,\tilde{w}_{\tau}\rangle- \langle v,\tilde{w}\rangle
&= \ (v_i - v_j)(\tilde{w}_j - \tilde{w}_i) \\
& \leq 0 
\end{split}
\end{equation}
Since $\langle v,\tilde{w}\rangle$ is already realizing the minimum, $[v]_{\Omega}$, this inequality must be an equality. Hence $\|v\|_{\Omega} = \langle v,\tilde{w}_{\tau}\rangle$ and we can proceed in this manner to obtain the desired ordered
$w \in \Omega$
\end{proof}

Now define the {\em backwards transfer} map $\mathcal{B} \colon \cev{C}(k,n) \to \cev{C}(k,n)$ by
$$v=(\underbrace{v_1, \dots v_1}_{t-times}, \dots , \underbrace{v_n, \dots v_n}_{T-times}) \mapsto  
\begin{cases}
      (\underbrace{v_1, \dots v_1}_{(t-1)-times}, v_1-1 \dots , v_n+1, \underbrace{v_n, \dots v_n}_{(T-1)-times}) &\text{ if $v_1 >v_n+1$}, \\
      v &  \text{ otherwise}.
\end{cases}
$$
We then have $\mathcal{B}^j(v)= \check{V}(k,n)$ for all sufficiently large $j$.

\begin{lemma}
\label{transfer-c}
For all $v \in \cev{C}(k,n)$ we have  $[\mathcal{B}(v)]_{\Omega}  \geq [v]_{\Omega} $.
\end{lemma}
\begin{proof}
By Lemma \ref{real-c} there is an ordered $w \in \Omega$ such that   $[\mathcal{B}(v)]_{\Omega} = \langle \mathcal{B}(v),w \rangle$. Hence,
\begin{equation}
\begin{split}
[v]_{\Omega} - [\mathcal{B}(v) ]_{\Omega} 
& \leq \ \langle v, w \rangle - \langle \mathcal{B}(v), w  \rangle \\
& = \ (v_{1}w_{t} + v_{n}w_{n-T}) - ((v_{1} - 1) w_{t} + (v_{n} + 1)w_{n-T}) \\
&= \ w_{t} - w_{n-T}\\
\end{split}
\end{equation}
which is nonpositive since $w$ is ordered.
\end{proof}

Finally, given $c_k(X_{\Omega}) = [v]_{\Omega}$, Lemma \ref{transfer-c} implies that, for all sufficiently large $j$, $$c_K(X_{\Omega}) = [\mathcal{B}^j(v)]_{\Omega}= [\check{V}(k,n)]_{\Omega}.$$

\section{New capacity computations}\label{examples}

In this section we use the formulas  of Propositions \ref{balance} and \ref{balance-c} to derive explicit formulas for the Gutt-Hutchings capacities in several new examples. 

\subsection{Graphs for $n=2$.} We start with simple but illuminating case when the domain $\Omega$ is  defined by the graph of a nice function of one variable. Let $g \colon [0,\lambda] \to \RR_{\geq 0}$ be a piecewise-smooth nonincreasing function such that $g(0)>0$ and $g(\lambda)=0$. Denote by  $\Omega_g$  the domain in $\RR^2_{\geq 0}$ that is bounded by the axes and the graph of $g$ and  set $X_g= \mu^{-1}(\Omega_g)$. 

By the Regular Value theorem we have the following.
 \begin{lemma}\label{reg}
If $g$ is smooth and $g'$ is finite and bounded away from zero, then the boundary of $X_g$ is smooth.
\end{lemma}

\subsubsection{Convex graph domains}
For $\lambda>0$, let $\widehat{\mathcal{V}}(\lambda)$ be the set of continuous functions $f\colon [0,\lambda] \to \RR_{\geq 0}$ with the following properties:
\begin{enumerate}
\item[(f1)] $f$ is smooth on $[0,\lambda)$,
\item[(f2)] $f(0)=\lambda$,
\item[(f3)] $f'(0) \leq 0$,
\item[(f4)] $f''<0$ on $[0, \lambda)$,
\item[(f5)] $f^{-1} = f$.
\end{enumerate}
Properties (f1)-(f4) imply that the toric domain $X_f$ is convex. Property (f5) implies that it is also symmetric. 
%Note also that each $f \in \widehat{\mathcal{V}}(\lambda)$ has a unique fixed point $x(f)  \in (0,1)$ and by (f5)
%we have
%\begin{equation*}
%\label{ }
%f'(x(f)) =-1.
%\end{equation*}
In this setting, Theorem  \ref{balance} yields the following.
\begin{proposition}\label{sym} Suppose $f$ is in $\widehat{\mathcal{V}}(\lambda)$. For even values of $k$,
\begin{equation*}
\label{subset}
c_k (X_f) = k x(f)
\end{equation*}
where $x(f)$ is the unique fixed point of $f$.
 For odd values of $k$, 
 \begin{equation*}
\label{subset}
c_k (X_f) = 
\begin{cases}
    \frac{k-1}{2}x_k + \frac{k+1}{2}f(x_k),  & \text{if} \,\, -\frac{k-1}{k+1} < f'(0)\\
     {} & {}\\
\frac{k+1}{2}\lambda, & \text{otherwise},
\end{cases}
\end{equation*}
where $x_k$ is defined uniquely by the equation
\begin{equation}
\label{hit}
f'(x_k)= -\frac{k-1}{k+1}.
\end{equation}
\end{proposition}

\begin{proof}
For even values of $k$, we have 
\begin{equation*}
\label{ }
V(k,2) = \left(\frac{k}{2},  \frac{k}{2} \right).
\end{equation*}
Theorem  \ref{balance} then yields
\begin{eqnarray*}
c_k(X_f) & = & \max_{x \in [0,\lambda]} \left\{\frac{k}{2} x+\frac{k}{2}f(x) \right\}\\
{} & = & \frac{k}{2} \max \left\{ \lambda, 2x(f) \right\}\\
{} & = & k x(f).
\end{eqnarray*}
Here we have used (f1)-(f5) to infer that $x(f)$ is the unique solution of $f'(x) =-1$, and that $x(f)> \lambda / 2$.

For odd values of $k$ we have 
\begin{equation*}
\label{ }
V(k,2) = \left( \left\lfloor \frac{k}{2} \right\rfloor,\left\lceil \frac{k}{2}\right\rceil  \right) = \left( \frac{k-1}{2} , \frac{k+1}{2}  \right)
\end{equation*}
and Theorem  \ref{balance} yields
\begin{equation*}
\label{ }
c_k(X_f) = \max_{x \in [0,\lambda]} \left\{\frac{k-1}{2} x +\frac{k+1}{2}  f(x)\right\}.
\end{equation*}
Set \begin{equation*}
\label{ }
\tilde{f}_k(x) =\frac{k-1}{2} x +\frac{k+1}{2}  f(x)
\end{equation*}
It is straightforward to check that if    
\begin{equation*}
\label{ }
f'(0) > -\frac{k-1}{k+1}
\end{equation*}
then we have $\tilde{f}_k(0)>0$ and, by (f5), $\tilde{f}_k(\lambda)<0$. Together with (f4), this implies that $\tilde{f}_k$ takes its maximum value at its unique critical point $x_k \in (0,\lambda)$ defined by
\begin{equation*}
\label{ }
f'(x_k) = -\frac{k-1}{k+1}.
\end{equation*}
If, on the other hand, 
\begin{equation*}
\label{ }
f'(0) \leq -\frac{k-1}{k+1},
\end{equation*}
then $\tilde{f}_k$ is strictly decreasing and  its maximum value is $\tilde{f}_k(0)=\frac{k+1}{2}\lambda$. 
\end{proof}

\subsubsection{Concave graph domains} For $\lambda>0$, let $\widehat{\mathcal{C}}(\lambda)$ be the set of continuous functions $h\colon [0,\lambda] \to \RR_{\geq 0}$ with the following properties:
\begin{enumerate}
\item[(h1)] $h$ is smooth on $(0,\lambda]$,
\item[(h2)] $h(\lambda) = 0$,
\item[(h3)] $h'(\lambda) \leq 0$,
\item[(h4)] $h''>0$ on $(0, \lambda]$, 
\item[(h5)] $h^{-1} = h$.
\end{enumerate}
The toric domain $X_h= \mu^{-1}(\Omega_h)$ is concave and symmetric. As before, each $h \in \widehat{\mathcal{C}}(\lambda)$  has a unique fixed point $x(h)  \in (0,1)$ and 
\begin{equation*}
\label{ }
h'(x(h)) =-1.
\end{equation*}

Theorem  \ref{balance-c} implies the following.
\begin{proposition}\label{sym-g} Suppose $h$ is in $\widehat{\mathcal{C}}(\lambda)$. For odd values of $k$,
\begin{equation*}
\label{subset}
c_k (X_h) = (k+1) x(h)
\end{equation*}
where $x(h)$ is the unique fixed point of $h$. For even values of $k$, 
  \begin{equation*}
\label{subset}
c_k (X_h) = 
\begin{cases}
    \frac{k+2}{2}\check{x}_k + \frac{k}{2}h(\check{x}_k),  & \text{if} \,\, -\frac{k}{k+2} < h'(\lambda)\\
     {} & {}\\
\frac{k}{2}\lambda, & \text{otherwise},
\end{cases}
\end{equation*}
where $\check{x}_k$ is the number that is defined uniquely by the condition
\begin{equation}
\label{hitcheck}
h'(\check{x}_k)= -\frac{k+2}{k}.
\end{equation}
\end{proposition}

%\medskip
%
%\noindent{\bf Observation 2.} For $h \in \widehat{\mathcal{C}}(\lambda)$ the capacities of $X_h$ are all determined by the values of $h$ at $0$, at $x(h)$, and at the  sequence of points $\check{x}_k$, for $k$ even,  that is defined by \eqref{hitcheck}  and monotonically increases to $x(h)$. 
%
%\medskip

\begin{proof}
For odd values of $k$, we have 
\begin{equation*}
\label{ }
\check{V}(k,2) = \left(\frac{k+1}{2},  \frac{k+1}{2} \right).
\end{equation*}
In this case, Theorem  \ref{balance-c}  together with conditions (h1)-(h5), yields
\begin{eqnarray*}
c_k(X_h)  & = & \min_{x \in [0,\lambda]} \left\{\frac{k+1}{2} x+\frac{k+1}{2}h(x) \right\}\\
{} & = & \frac{k+1}{2} \min \left\{ \lambda, 2x(h) \right\}\\
{} & = & (k+1)x(h).
\end{eqnarray*}

For even values of $k$ we have 
\begin{equation*}
\label{ }
\check{V}(k,2) = \left( \left\lceil \frac{k+1}{2}\right\rceil , \left\lfloor \frac{k+1}{2} \right\rfloor \right)= \left( \frac{k+2}{2} , \frac{k}{2}  \right)
\end{equation*}
and Theorem  \ref{balance-c} yields
\begin{equation*}
\label{ }
c_k(X_h) = \min_{x \in [0,\lambda]} \left\{\frac{k+2}{2} x + \frac{k}{2} h(x)\right\}.
\end{equation*}
Arguing as in the proof of Proposition \ref{sym} one can show that if    
\begin{equation*}
\label{ }
h'(0) \geq -\frac{k+2}{k},
\end{equation*}
or equivalently if $h'(\lambda) \leq -\frac{k}{k+2}$,  then the function
\begin{equation*}
\label{ }
x \mapsto \frac{k+2}{2} x + \frac{k}{2} h(x)
\end{equation*}
 has no critical points in $(0,\lambda)$ and a minimum value of $\frac{k}{2}\lambda$. Otherwise, it has a unique critical point $\check{x}_k \in (0,\lambda)$ defined by
\begin{equation*}
\label{ }
h'(\check{x}_k) = -\frac{k+2}{k}
\end{equation*}
at which it takes its minimum value.
\end{proof}

%\begin{remark}
%When one considers a family of graph domains, say $X_{g_s}$ for $s \in [-1,1]$, the conditional definitions of the odd capacities in Proposition \ref{sym} and the even capacities in Proposition \ref{sym-g} leads to changes in behavior for functions of the type $s \mapsto c(X_{g_s})$ for some capacity $c$. This occurs  below in Example \ref{lp}  and Example \ref{diagonal}. A similar change in behavior is  observed by Ostrover and Ramos, in \cite{or}, for the outer ball capacity of the  $\ell_p$-sum of Lagrangian disks. There it is referred to as a {\em phase change}. 
%\end{remark}

%\begin{example}
%Consider $h(x)=(1-x^{\frac{1}{2}})^2$. Here $x(h) =\frac{1}{4}$, $\lambda=1$, $h'(1)=0$ and $$\check{x}_k = \left(\frac{k}{2(k+1)}\right)^2.$$ %$f(x_n(f)) = \frac{n+1}{\sqrt{(n+1)^2 + n^2}}$
%%$$
%%nx_n(f) + (n+1)f(x_n(f)) = \sqrt{(n+1)^2 + n^2}>n+1
%%$$
%By Proposition \ref{sym-g}, we then have  
%\begin{equation}
%\label{ }
%c_k (X_{g}) = 
%\begin{cases}
%    \frac{k(k+2)}{4(k+1)}  &, \text{ if $k$ is even}, \\
%    {} & {}\\
%   \frac{k+1}{4}    &, \text{ if $k$ is odd}.
%\end{cases} 
%\end{equation}   
%\end{example}

\begin{example}\label{lp} Motivated by the work of Ostrover and Ramos in \cite{or} we apply Proposition \ref{sym} and Proposition \ref{sym-g} to the  $\ell_p$-sum of two Lagrangian disks, for $1 \leq p \leq \infty$. These domains are defined as follows
$$
P_L(p)= \left\{ (x_1+i y_1, x_2+ iy_2) \in \CC^2 \mid ( x_1^2 +x_2^2)^{\frac{p}{2}}+ ( y_1^2 +y_2^2)^{\frac{p}{2}} \leq 1\right\}.
$$

\begin{proposition}\label{lp} 
 For $1\leq p \leq 2$, we have 
\begin{equation*}
\label{ }
c_k(P_L(p))=
\begin{cases}
   k \frac{2\Gamma(1+\frac{1}{p})^2}{\Gamma(1+\frac{2}{p})}   & \text{ for even $k$}, \\
   \frac{k+1}{2}\frac{2\pi}{4^{1/p}}   & \text{ for odd $k < \frac{1}{\sqrt{\frac{2}{p}}-1}$}, \\
  kg_p(-v_k) + (k+1)\pi v_k     & \text{for odd  $k \geq \frac{1}{\sqrt{\frac{2}{p}}-1}$},
\end{cases}
\end{equation*}
where 
\begin{equation}
\label{gp}
g_p(v) =2\int_{(\frac{1}{2} -\sqrt{\frac{1}{4} -v^p})^{1/p}}^{(\frac{1}{2} +\sqrt{\frac{1}{4} -v^p})^{1/p}} \sqrt{(1-r^p)^{2/p} -\frac{v^2}{r^2}} \, dr.
\end{equation}
and $v_k$ is defined by \begin{equation}
\label{vk}
g'_p(-v_k) =-\pi \frac{k+1}{k}.
\end{equation}

For $2<p<\infty$, we have 
\begin{equation*}
\label{ }
c_k(P_L(p))=
\begin{cases}
(k+1) \frac{2\Gamma(1+\frac{1}{p})^2}{\Gamma(1+\frac{2}{p})}   & \text{ for odd $k$}, \\
 \frac{k}{2}\frac{2\pi}{4^{1/p}}   & \text{ for even $k <  \frac{\sqrt{\frac{2}{p}}}{1-\sqrt{\frac{2}{p}}}$}, \\
  (k+2)\pi \check{v}_k+ (k+1)g_p(-\check{v}_k)    & \text{for even  $k \geq \frac{\sqrt{\frac{2}{p}}}{1-\sqrt{\frac{2}{p}}}$},
\end{cases}
\end{equation*}
where $g_p$ is as above and $\check{v}_k$ is defined by 
\begin{equation}
\label{vkk}
g'_p(-\check{v}_k) =-\pi \frac{k}{k+1}.
\end{equation}
\end{proposition}

\begin{proof}

The crucial starting point is the following result from \cite{or}.

\begin{theorem}[Theorem 5, \cite{or}]
The interior of $P_L(p)$ is symplectomorphic to the interior of the toric domain $X_p =\mu^{-1}(\Omega_p) \subset \RR^4$ where $\Omega_p \subset \RR^2_{\geq0}$ is the region bounded by the axes and the curve
\begin{equation*}
\label{ }
\alpha_p(v)=
\begin{cases}
    (g_p(-v),-2\pi v+g_p(-v))  &, \text{ for $v \in [-(1/4)^{1/p},0]$} \\
    {} & {}\\
 (2\pi v+ g_p(v), g_p(v))    &, \text{ for $v \in [0,(1/4)^{1/p}]$}
\end{cases}
\end{equation*}
where $g_p(v)$ is defined as in \eqref{gp}.
\end{theorem}

In the process of computing the first two ECH capacities of $X_p$, Ostrover and Ramos also prove the following.

\begin{lemma}[Lemma 18, \cite{or}] For all  for $p\geq1$:
\begin{itemize}
  \item[(i)] $g_p(0) =\frac{2\Gamma(1+\frac{1}{p})^2}{\Gamma(1+\frac{2}{p})}$ and $g_p(1/4^{1/p})=0$.\\
  \item[(ii)] $g_p$ is strictly decreasing.\\
  \item[(iii)] $\displaystyle{\lim_{v\to 0}g'_p(v) =-\pi}$ and $\displaystyle{\lim_{v\to 1/4^{1/p}}g'_p(v) =-\sqrt{\frac{2}{p}}\pi}$.
\end{itemize}
Moreover, for $1< p < 2$ the function $g_p$ is strictly concave and for $p >2$ the function $g_p$ is strictly convex.
\end{lemma}

\noindent{\bf Case 1: $1\leq p \leq 2$.} In this range, the toric region  $X_p$ is  convex and symmetric. Let $f_p$ be the function whose graph corresponds to the image of $\alpha_p$. Then $$x(f_p) = g_p(0)=\frac{2\Gamma(1+\frac{1}{p})^2}{\Gamma(1+\frac{2}{p})}.$$ As well, $-\frac{k-1}{k+1} \leq f_p'(0)$ if and only if 
$k \geq \frac{1}{\sqrt{\frac{2}{p}}-1}$. If this holds, then we have 
$$
x_k(f_p) = g_p(-v_k)
$$
where $v_k \in [0,(1/4)^{1/p}]$ is determined uniquely by the equation
\begin{equation}
\label{vk}
g'_p(-v_k) =-\pi \frac{k+1}{k}.
\end{equation}
With this, the first capacity formula in Proposition \ref{lp} follows immediately from Proposition \ref{sym}.

\noindent{\bf Case 2: $p>2$.} In this range, the toric region  $X_p$ is  convcave and symmetric. Let $h_p$ be the function whose graph corresponds to the image of $\alpha_p$. As before,  $$x(h_p) = g_p(0)=\frac{2\Gamma(1+\frac{1}{p})^2}{\Gamma(1+\frac{2}{p})}.$$ Now, we have $-\frac{k}{k+2} \leq h_p'(\frac{2\pi}{4^{1/p}})$ if and only if $k\geq \frac{\sqrt{\frac{2}{p}}}{1-\sqrt{\frac{2}{p}}}$. In this case
$$
\check{x}_k(h_p) = 2\pi \check{v}_k+ g_p(\check{v}_k)
$$
where $\check{v}_k \in [-(1/4)^{1/p},0]$ is determined uniquely by the equation
\begin{equation}
\label{vkk}
g'_p(-\check{v}_k) =-\pi \frac{k}{k+1}.
\end{equation}
The second capacity formula in Proposition \ref{lp} now follows  from Proposition \ref{sym-g}.

\end{proof}
\end{example}

\subsection{Symplectic $\ell^p$-sums (Graphs for $n>2$)}
One can also consider toric domains $X_g$ with projections $\Omega_g \subset \RR^n$, for $n>2$, that are determined by the graph of a function $g \colon \Lambda \subset \RR^{n-1} \to \RR_{\geq 0}$. The corresponding capacity formulas may involve many cases. A simple but interesting set of  examples of this type are the symplectic $\ell^p$-sums
$$
B_p^n= \left\{ (z_1, \dots z_n) \in \CC^n \mid  \pi^{p/2}(|z_1|^{p} + \dots + |z_n|^p) \leq 1\right\},
$$
for $p>0$. Note that $B^n_p = \mu^{-1}(\Omega^n_p)$ where $\Omega^n_p \subset \RR^n_{\geq 0}$ is the region bounded by the coordinate hyperplanes and the graph 
%$$\{(x_1, \dots, x_{n-1},f^n_{p}(x_1, \dots, x_{n-1})) \mid (x_1, \dots, x_{n-1})\in \Lambda^{n-1}_{p/2})\}$$ 
of the function
 $$
 f^n_{p}(x_1, \dots, x_{n-1})=(1-x_1^{p/2} - \dots - x_{n-1}^{p/2})^{\frac{2}{p}}
$$
with domain
$$
\Lambda_{p/2}^{n-1}= \left\{ (x_1, \dots x_n) \in \RR^{n-1}_{\geq 0} \mid x_1^{p/2} + \dots + x_{n-1}^{p/2} \leq 1\right\}.
$$
Each  toric domain $B^n_p$ is symmetric and $B^n_2$ is the standard closed unit ball.

\begin{proposition}\label{sym-sum}
For $p>2$, we have 
\begin{equation}
\label{bpn>2}
c_k(B^n_p) = \bigg( \left(n- (k \mod n)\right)\left\lfloor \frac{k}{n} \right\rfloor ^{\frac{p}{p-2}} + \Big(k \mod n\Big) \left\lceil \frac{k}{n} \right\rceil ^{\frac{p}{p-2}}\bigg)^{\frac{p-2}{p}}.
\end{equation}
For $0<p<2$, we have
\begin{equation}
\label{bpn<2}
c_k(B^n_p)=\bigg( (k' \mod n) \left\lceil \frac{k'}{n} \right\rceil ^{\frac{p}{p-2}} + \Big(n - (k' \mod n)\Big) \left\lfloor \frac{k'}{n} \right\rfloor ^{\frac{p}{p-2}}\bigg)^{\frac{p-2}{p}},
\end{equation}
where $k' = k+n-1$. 
\end{proposition}

\begin{remark}
For $n=2$ and $p>2$, equation \eqref{bpn>2} simplifies to 
\begin{equation}
\label{simple}
c_k (B^2_p) = 
\begin{cases}
    \frac{k}{2^{\frac{2}{p}}}, & \text{for even  $k$} \\
    {} & {}\\
   \left( \left(\frac{k+1}{2}\right)^{\frac{p}{p-2}}+\left(\frac{k-1}{2}\right)^{\frac{p}{p-2}} \right)^{\frac{p-2}{p}},    & \text{for odd $k$}.
\end{cases}
\end{equation}
As $p \to \infty$, $B^2_p$  converges to the polydisk $P(1,1)$ in the Hausdorff topology and equation \eqref{simple} yields $$\lim_{p \to \infty} c_k (B^2_p) =k$$ which agrees with the formula for  $c_k(P(1,1))$ from \cite{gh}. Similarly,  as $p \to 2^+$ the domains  $B^2_p$  converge to the ball $E(1,1)$ and  \eqref{simple} yields  
$$\lim_{p \to 1^+} c_k (B^2_p) =\begin{cases}
    \frac{k}{2} &, \text{ for even $k$}, \\
    {} & {}\\
  \frac{k+1}{2} &, \text{ for odd $k$}.
\end{cases}
$$
which agrees with the formula for  $c_k(E(1,1))$ from \cite{gh}.

For the case $n=2$ and $p<2$, equation \eqref{bpn<2} simplifies to 
\begin{equation}
\label{simple2}
c_k (B^2_p) = 
\begin{cases}
    \frac{k + 1}{2^{\frac{2}{p}}}, & \text{for odd  $k$} \\
    {} & {}\\
   \left( \left(\frac{k}{2}\right)^{\frac{p}{p-2}}+\left(\frac{k+2}{2}\right)^{\frac{p}{p-2}} \right)^{\frac{p-2}{p}},    & \text{for even $k$}.
\end{cases}
\end{equation}
As $p \to 2^-$ the domains  $B^2_p$  converge to the ball $E(1,1)$ and  \eqref{simple2} yields  
$$\lim_{p \to 1^+} c_k (B^2_p) =\begin{cases}
    \frac{k+1}{2} &, \text{ for odd $k$}, \\
    {} & {}\\
  \frac{k}{2} &, \text{ for even $k$}.
\end{cases}
$$
which again agrees with the formula for  $c_k(E(1,1))$ from \cite{gh}.\ \\

\end{remark}

\begin{proof}

\noindent{\bf Case 1:  $p>2$.} In this case $B^n_p$ is (strictly) convex  and so by Theorem  \ref{balance} 
\begin{eqnarray*}
c_k(X^n_p)& = & \max \left\{ \sum_{i=1}^{n-1} V(k,n)_i x_i   + V(k,n)_n (1-x_1^{p/2} - \dots - x_{n-1}^{p/2})^{\frac{2}{p}} \mid (x_1, \dots, x_{n-1}) \in \Lambda_{p/2}^{n-1} \right\}.
\end{eqnarray*}
(For the moment it is useful to forget that we know the components $V(k,n)_i$ of $V(k,n)$.)

The function $$F(x_1, \dots, x_{n-1}) =\sum_{i=1}^{n-1} V(k,n)_i x_i   + V(k,n)_n (1-x_1^{p/2} - \dots - x_{n-1}^{p/2})^{\frac{2}{p}}$$ attains it maximum value at its unique  critical point which we now solve for. The equation $\frac{\partial F}{\partial x_i}=0$ is equivalent to 
\begin{equation*}
\label{ }
x_i^{p/2} = \left(\frac{V(k,n)_i}{V(k,n)_n}\right)^{\frac{p}{p-2}}\left(1-x_1^{p/2} - \dots - x_{n-1}^{p/2}\right).
\end{equation*}
Setting $y_i = x_i^{p/2}$ and $\alpha_i = \left(\frac{V(k,n)_i}{V(k,n)_n}\right)^{\frac{p}{p-2}}$ we can rewrite this as 
\begin{equation*}
\label{ }
\alpha_i y_1 + \dots + (\alpha_i +1)y_i + \dots +\alpha_iy_{n-1} = \alpha_i.
\end{equation*}
Hence the critical point of $F$ corresponds to the unique solution of the 
linear system
\begin{equation*}
\label{ }
\left(\mathbb{1}_{n-1} + \vec{\alpha} \cdot (\vec{1})^T\right)\vec{y} = \vec{\alpha}
\end{equation*}
where $(\vec{1})^T = (1,1,\dots,1)$. Using the Sherman-Morrison formula we arrive at the expression
\begin{equation*}
\label{ }
\vec{y} = \left(\frac{1}{1+\sum_{i=1}^{n-1}\alpha_i}\right) \vec{\alpha}.
\end{equation*}
Hence
\begin{eqnarray*}
x_j & = & \left(1+\sum_{i=1}^{n-1}\left(\frac{V(k,n)_i}{V(k,n)_n}\right)^{\frac{p}{p-2}} \right)^{\frac{-2}{p}}  \left(V(k,n)_n\right)^{\frac{-2}{p-2}} \left(V(k,n)_j\right)^{\frac{2}{p-2}}\\
{} &= &\left( \sum_{i=1}^{n} \left( V(k,n)_i \right)^{\frac{p}{p-2}} \right)^{\frac{-2}{p}}\left(V(k,n)_j\right)^{\frac{2}{p-2}}.
\end{eqnarray*}
%\textcolor{blue}{Here we observe that $b(k, n) = \left( \sum_{i = 1}^{n - 1} \left(V(k,n)_i\right)^{\frac{p}{p-2}} + \left(V(k,n)_n\right)^{\frac{2}{p-2}} \right)^{\frac{-2}{p}} = \left( \sum^n_{i=1} \left( V(k,n)_i \right)^{\frac{p}{p-2}} \right)^{\frac{-2}{p}}$}
Evaluating $F$ at this critical point we get
%\begin{eqnarray}
%\label{ }
%c_k(B^n_p) &=& \left( \sum^{n}_{i = 1} \left( V(k,n)_i \right)^{\frac{p}{p-2}} \right)^{\frac{p-2}{p}} \\
%{}&=& \bigg( \left(n- (k \mod n)\right)\left\lfloor \frac{k}{n} \right\rfloor ^{\frac{p}{p-2}} + \Big(k \mod n\Big) \left\lceil \frac{k}{n} \right\rceil ^{\frac{p}{p-2}}\bigg)^{\frac{p-2}{p}}
%\end{eqnarray}
\begin{equation}
c_k(B^n_p) = \bigg( \left(n- (k \mod n)\right)\left\lfloor \frac{k}{n} \right\rfloor ^{\frac{p}{p-2}} + \Big(k \mod n\Big) \left\lceil \frac{k}{n} \right\rceil ^{\frac{p}{p-2}}\bigg)^{\frac{p-2}{p}},
\end{equation}
as desired.
%\textcolor{blue}{    It would be easier to calculate if we do not substitute the expression of $\alpha_i$ and $y_i$. The final expression of the capacity I obtain is: $\left( \sum^n_{i = 1} \left( V_i^k \right)^{\frac{p}{p-2}} \right)^{\frac{p-2}{p}} $   This final result coincide with the special case of n=2 in the following expression. Furthermore, I believe that the expression applies both to concave and convex region as the rationale behind both result is setting all the partial derivative to 0. }\ \\
 
 \noindent{\bf Case 2:  $0<p<2$.} Here,  $B_p^n$ is concave, and 
 \begin{eqnarray*}
c_k(X^n_p) & = & \min \left\{ \sum_{i=1}^{n-1} \check{V}(k,n)_i x_i   + \check{V}(k,n)_n (1-x_1^{p/2} - \dots - x_{n-1}^{p/2})^{\frac{2}{p}} \mid (x_1, \dots, x_{n-1}) \in \Lambda_{p/2}^{n-1} \right\}.
\end{eqnarray*}
Arguing as above, it follows that the function $$\sum_{i=1}^{n-1} \check{V}(k,n)_i x_i   + \check{V}(k,n)_n (1-x_1^{p/2} - \dots - x_{n-1}^{p/2})^{\frac{2}{p}}$$ has a unique global minimum at 
\begin{equation}
x_j =   \left( \sum_{i=1}^{n} \left( \check{V}(k,n)_i \right)^{\frac{p}{p-2}} \right)^{\frac{-2}{p}} \left(\check{V}(k,n)_j\right)^{\frac{2}{p-2}}.
\end{equation}
Hence
\begin{equation}
c_k(B^n_p)=\bigg( (k' \mod n) \left\lceil \frac{k'}{n} \right\rceil ^{\frac{p}{p-2}} + \Big(n - (k' \mod n)\Big) \left\lfloor \frac{k'}{n} \right\rfloor ^{\frac{p}{p-2}}\bigg)^{\frac{p-2}{p}},
\end{equation}
where $k' = k+n-1$. 

\end{proof}

\subsection{Symmetric toric polytopes} Let $P$ be a convex polytope in $\RR^n_{\geq 0}$ with verticies $\{p_j\}_{j=1,\dots,N}$. Note that $P$ is symmetric  if every permutation matrix of $\RR^n$ maps the set of vertices of $P$ onto itself. 

\begin{proposition}\label{poly}
If $P$ is symmetric and the toric domain $\mu^{-1}(P)$ is convex then 
\begin{equation}
\label{convexpoly}
c_k(\mu^{-1}(P))= \max_j \{ \langle V(k,n), p_j \rangle\}.
\end{equation}
\end{proposition}

\begin{proof}
This follows from Theorem  \ref{balance}  and standard linear programming.
\end{proof}

\begin{remark}The simplicity of this formula for the Gutt-Hutchings (conjecturally Ekeland-Hofer) capacities of this class of star-shaped domains in any even dimension is quite noteworthy. Especially since any symmetric convex toric domain can be $C^0$-approximated by such a domain. The formula is used to provide models in Section \ref{open}.
\end{remark}

\begin{example}\label{diagonal}
For $r \in [1/2,1]$, let $\Omega_r$ be the symmetric convex polytope in $\RR^2_{\geq 0}$ with vertices
$$
\{(0,0), (1,0), (0,1), (r,r)\}.
$$
 For even values of $k$,  \eqref{convexpoly} implies that  
\begin{equation*}
\label{ }
c_k(\mu^{-1}(\Omega_r)) = kr.
\end{equation*}
For odd values of $k$,  \eqref{convexpoly} implies that
\begin{eqnarray*}
c_k(\mu^{-1}(\Omega_r)) & = & \max \left\{ \frac{k+1}{2}, \,kr\right\}.
\end{eqnarray*}
This  piecewise linear function of $r$ is constant and equal to $ \frac{k+1}{2}$ for $r \in \left[\frac{1}{2},\frac{k+1}{2k}\right]$,  and then grows linearly as $kr$.
\end{example}

%\begin{definition}
%Let $P\subset \RR^n_{\geq 0}$ be a symmetric toric polytope such that $\mu^{-1}(P)$ is convex. 
%We say that $P$ has {\em smoothly realizable capacities} if for every $\epsilon>0$ there is a domain $\Omega_{\epsilon}\subset \RR^n_{\geq 0}$ with smooth boundary such that  
%\begin{enumerate}
%  \item $\mu^{-1}(\Omega_{\epsilon})$ is convex
%  \item $\Omega_{\epsilon}$ is within $\epsilon$ of $P$ with repect to the Hausdorff metric.
%  \item $c_k(\mu^{-1}(\Omega_{\epsilon}))= c_k(\mu^{-1}(P))$ for all $k \in \NN$.
%\end{enumerate}
%\end{definition}

\section{The proofs of Theorems   \ref{novolume}, \ref{mutual}, and \ref{blind}}\label{proofs}

\subsection{Preliminaries} Here we will use the following useful simplifications of Propositions \ref{sym} and \ref{sym-g}.
\begin{corollary}
\label{f}
Let $f$ be a function in $\widehat{\mathcal{V}}(\lambda)$ such that $f'(0) \in (-1/2,0)$. Then $X_f$ is a (strictly) convex toric domain in $\RR^4$ with a smooth boundary. For even values of $k$ we have 
\begin{equation}
c_k (X_f) = k x(f),
\end{equation}
where $x(f)$ is the unique fixed point of $f$. The first capacity of $X_f$ is equal to $\lambda$, and for odd $k>1$ we have 
 \begin{equation}
c_k (X_f) = \frac{k-1}{2}x_k + \frac{k+1}{2}f(x_k)
\end{equation}
where $x_k$ is the unique solution of
\begin{equation}
\label{hit}
f'= -\frac{k-1}{k+1}.
\end{equation}
The sequence of points $x_k$ increases monotonically and converges to $x(f)$.
\end{corollary}

\begin{corollary}
\label{h}
Let $h$ be a function in $\widehat{\mathcal{C}}(\lambda)$ such that $h'(\lambda) \in (-1/2,0)$. Then $X_h$ is a strictly concave toric domain in $\RR^4$ with a smooth boundary. For odd values of $k$,
\begin{equation}
\label{subset}
c_k (X_h) = (k+1) x(h)
\end{equation}
where $x(h)$ is the unique fixed point of $h$. For even values of $k$ we have  
 \begin{equation*}
\label{}
c_k (X_h) = \frac{k+2}{2}\check{x}_k + \frac{k}{2}h(\check{x}_k)
\end{equation*}
where $\check{x}_k$ is the unique solution of 
\begin{equation}
\label{hitchecks}
h'= -\frac{k+2}{k}.
\end{equation}
The sequence of points $\check{x}_k$ increases monotonically and converges to $x(h)$.
\end{corollary}

\subsubsection{Symmetric perturbations}\label{pert}
Let $g$ be a function in either $\widehat{\mathcal{V}}(\lambda)$ or $\widehat{\mathcal{C}}(\lambda)$. Let $\beta \colon [0,\lambda] \to \RR$ be a smooth function with support $[a,b]$ contained in $(0,x(g))$. Then for all sufficiently small $\epsilon>0$ the restriction of the function $g+\epsilon \beta$ to $[0,x(g)]$ can be extended to a unique function in the same set $\widehat{\mathcal{V}}(\lambda)$ or $\widehat{\mathcal{C}}(\lambda)$ as $g$.  This (symmetric) extension of $g$ is of the form  $g+\epsilon (\beta+\tilde{\beta})$ for a unique function $\tilde{\beta} \colon [0,\lambda] \to \RR$ with support $[g(b),g(a)] \subset (x(g), \lambda)$. A simple computation, using integration by parts, implies that 
\begin{equation}
\label{sameint}
\int_0^{\lambda} \beta = \int_0^{\lambda} \tilde{\beta}.
\end{equation}

\subsection{Proof of Theorem \ref{novolume}}
Let $f$ be a function in $\widehat{\mathcal{V}}(1)$ as in Corollary \ref{f}. For an odd integer $j>1$ choose a smooth bump function $\eta_j \colon [0,1] \to \RR_{\geq 0}$ such that  
\begin{itemize}
  \item[(i)] the support of $\eta_j$ is contained in $(x_j, x_{j+2})$.
  \item[(ii)] $\int_0^1 \eta_j =\frac{1}{2}.$
\end{itemize}
For all  $\delta \geq 0$ small enough the function $f+\delta(\eta_j+\tilde{\eta}_j)$ is in $\widehat{\mathcal{V}}(1)$ and we set $V_{\delta} = X_{f+\delta(\eta_j+\tilde{\eta}_j)}.$ Condition (i) and Corollary \ref{f} imply that $c_k(V_{\delta}) = c_k(V_0)$ for all $k \in \NN$. Condition (ii) and \eqref{sameint} imply that 
$$\mathrm{volume}(V_{\delta})=\mathrm{volume}(V_0) +\delta.$$

\subsection{Proof of Theorem \ref{mutual}}
For even values of $j$ the desired family of domains $V^j_{\delta}$ will be comprised of strictly concave toric domains. For odd values of $j$ it will be comprised of strictly convex toric domains.

\subsubsection{Even Indices}

Let $h$ be a function in $\widehat{\mathcal{C}}(1)$,  as in Corollary \ref{h},  with  $h'(1) \in (-1/2,0)$. For  $j \in 2\NN$, let $\beta_j \colon [0,1] \to \RR$ be a smooth function with the following three properties: 
\begin{itemize}
  \item[(i)] $\beta_j=\frac{2}{j}$ in a neighborhood of $\check{x}_j$,
  \item[(ii)] $\beta_j(x)$ is compactly supported in $(\check{x}_{j-2} ,\check{x}_{j+2})$,
  \item[(iii)] $\int_0^1 \beta_j =0.$
\end{itemize}
For all sufficiently small $\delta\geq 0$ the function $h+\delta (\beta_j+\tilde{\beta}_j)$ is in $\widehat{\mathcal{C}}(1)$. Restricting ourselves to such values of $\delta$ we set $V^j_{\delta} = X_{h+\delta (\beta_j+\tilde{\beta}_j)}.$ Properties (i) and (ii) imply that the solution of 
\begin{equation*}
\label{}
(h+\delta(\beta_j+\tilde{\beta}_j))'= -\frac{k+2}{k}
\end{equation*}
is the same for all for all even $k$ and all $\delta>0$. By  Corollary \ref{h} and (i), we then have 
\begin{eqnarray*}
c_j(V^j_{\delta}) & = & \frac{j+2}{2}\check{x}_j + \frac{j}{2}\left(h(\check{x}_j)+\delta (\beta_j(\check{x}_j)+\tilde{\beta}_j(\check{x}_j))\right) \\
{} & = & c_j(V^j_0) + \delta.
\end{eqnarray*}
On the other hand,  (ii) implies that
$c_k(V^j_{\delta}) =c_k(V^j_0)$ for all even $k \neq j$. Similarly, (ii) implies that $x(h+\delta (\beta_j+\tilde{\beta}_j)) =x(h)$  for all sufficiently small $\delta>0$, and so $c_k(V^j_{\delta}) =c_k(V^j_0)$ for all odd $k$ as well. Finally, it follows from (iii) and \eqref{sameint} that $$\mathrm{volume}(V^j_{\delta})=\mathrm{volume}(V^j_0).$$

\subsubsection{Odd Indices}
Here the argument is entirely similar except for a few details in the case $j=1$.  We describe this case and leave the others to the reader. Let $f$ be a function in $\widehat{\mathcal{V}}(1)$ as in Corollary \ref{f}. Let $\beta_1 \colon [0,1] \to \RR$ be a smooth function such that  
\begin{itemize}
  \item[(i)] $\beta_1=1$ in a neighborhood of $0$.
  \item[(ii)] $\beta_1(x) =0$ for all $x \geq x_3/2$. 
  \item[(iii)] $\int_0^1 \beta_1 =0.$
\end{itemize}
For  all  sufficiently small $\delta>0$, the restriction of the function $f+\delta\beta_1$ to $[0,x(f)]$ can be extended to a unique function in $ \widehat{\mathcal{V}}(1 +\delta)$ which we again denote by $f+\delta(\beta_1 +\tilde{\beta}_1)$. Set $V^1_{\delta} = X_{f+\delta (\beta_1+\tilde{\beta}_1)}.$

By (i), we then have $c_1(V^1_{\delta}) =1+\delta = c_1(V^1_0) +\delta$. It follows from (ii) that $x(V^1_{\delta}) =x(f)$ and so $c_k(V^1_{\delta})) =c_k(V^1_0))$ for all even $k$. As well, for odd values of $k>1$, property (ii) implies that the solution of 
\begin{equation}
(f+\epsilon(\beta_1 +\tilde{\beta}_1))'= -\frac{k-1}{k+1}.
\end{equation}
is identical to that of \eqref{hit}. Hence, $c_k(V^1_{\delta})) =c_k(V^1_0))$ for these values of  $k$ as well. Finally, it follow from (iii)  and \eqref{sameint} that $$\mathrm{volume}(V^1_{\delta}))=\mathrm{volume}(V^1_0)).$$

\subsection{Proof of Theorem \ref{blind}}
\begin{lemma}\label{ready}
 Let $h$ be a function in $\widehat{\mathcal{C}}(\lambda)$ with $h'(\lambda) \in (-1/2,0)$. Let  $\rho \colon [0,1] \to \RR_{\geq 0}$ be a nonconstant smooth function such that  
\begin{itemize}
  \item[(i)] the support of $\rho$  does not contain $x(h)$ or $\check{x}_k$ for and even $k \in \NN$.
  \item[(ii)] $\int_0^1 \rho =0.$
\end{itemize}
Then, for all sufficiently small $\delta>0$ the function $h+\delta(\rho+ \tilde{\rho})$ is in $\widehat{\mathcal{C}}(\lambda)$, we have  $$c_k(X_{h+\delta(\rho+ \tilde{\rho})})=c_k(X_{h})$$ for all $k \in \NN$,  and
$$\mathrm{volume}(X_{h+\delta(\rho+ \tilde{\rho})})=\mathrm{volume}(X_{h}).$$
\end{lemma}

\begin{proof}
The assertion about the capacities follows from (i) and Corollary \ref{h}. The assertion about the volumes follows from (ii) and \eqref{sameint}. 
\end{proof}

We will define explicit  functions $h$ and $\rho$ and prove that the symplectomorphism type of $X_{h+\delta(\rho+ \tilde{\rho})}$ varies with $\delta$ for all sufficiently small $\delta$. 
To do this we will use the Embedded Contact Homology capacities, $\{c_k^{ECH}\}_{k \in \NN}$,  defined by Hutchings in \cite{ech} and the formula for the ECH capacities of concave toric domains established by  Choi, Cristofaro-Gardiner, Frenkel, Hutchings and Ramos in \cite{choi}. In particular, we will prove, in Lemma \ref{capcomp} that 
\begin{equation*}
\label{capcomp}
c^{\mathrm{ECH}}_9(X_{h+\delta (\rho+ \tilde{\rho})}) =c^{\mathrm{ECH}}_9(X_{h})+\delta.
\end{equation*}

\begin{remark}
Given the results of \cite{chr}, a similar argument should work for any choices of $h$ and $\rho$, as above. In particular, it should be possible to show that $c^{\mathrm{ECH}}_k(X_{h+\delta(\rho+ \tilde{\rho})}) =c^{\mathrm{ECH}}_k(X_{h})+\delta$ for some $k \in \NN$.
\end{remark}

\begin{remark}
It is well known that the ECH capacities have their own blind spots corresponding to singular boundaries. In particular, $E(1,2)$ and $P(1,1)$ have the same ECH capacities (and volumes). It is not known whether star-shaped regions in $\RR^4$ with smooth boundaries and the same ECH capacities must be symplectomorphic.
\end{remark}

\medskip

\noindent{\bf ECH capacities of concave toric domains.} We recall here the formula from \cite{choi}  for the ECH capacities of a concave toric domain $X_{\Omega} \subset \RR^4$. 

\begin{theorem}[\cite{choi}] If $X_{\Omega}$ is a concave toric domain in $\RR^4$ and the ordered weight expansion of $\Omega$ is $$\vec{w}(\Omega) =\{w_1, w_2, w_3 \dots \},$$ then 
\begin{equation}
\label{choi}
c^{\mathrm{ECH}}_k(X_{\Omega})= \max\left\{  \sum_{i=1}^k d_iw_i \,\middle|\, \sum_{i=1}^k \frac{d_i^2+d_i}{2} \leq k, \, d_i \in \{0\} \cup \NN \right\}.
\end{equation}
\end{theorem}

To make use of this we must also recall the definition of the the ordered weight expansion $\vec{w}(\Omega)$ of a concave domain $\Omega$. We begin with the basic procedure of {\em concave subdivision} which derives three smaller (possibly empty) concave domains from $\Omega$. We denote  these by $T(\Omega)$, $\Omega_1$ and $\Omega_2$. 

Let $T(c)$ be the triangle in $ \RR^2_{\geq 0}$ with vertices 
$$
\left\{(0,0), (c,0), (0,c)\right\}.
$$
Setting $\tau(\Omega) = \max \{c \mid T(c) \subset \Omega \}, $ we define $$T(\Omega)=T(\tau(\Omega)).$$
Now $\Omega \smallsetminus T(\Omega)$ as the disjoint union of two, possibly empty, subsets whose closures we denote by $\tilde{\Omega}_1$ and  $\tilde{\Omega}_2$, If either of these sets is nonempty, then the labels are chosen so that  $\tilde{\Omega}_1$ is disjoint from the vertical axis and $\tilde{\Omega}_2$ is disjoint from the horizontal axis. If both these sets are nonempty their intersection is the point $(\tau(\Omega),\tau(\Omega))$. If $\tilde{\Omega}_1$ is nonempty then it has a unique obtuse angle and we define $\Omega_1$ to be the concave domain obtained by translating $\tilde{\Omega}_1$ by $(-\tau(\Omega),0)$, and applying the transformation $$\left(\begin{array}{cc}1 & 1 \\0 & 1\end{array}\right).$$
Similarly, if $\tilde{\Omega}_2$ is nonempty, then it has a unique obtuse angle and we define $\Omega_2$ be the concave domain obtained by translating $\tilde{\Omega}_2$ by $(0,-\tau(\Omega))$, and applying the transformation $$\left(\begin{array}{cc}1 & 0 \\1 & 1\end{array}\right).$$ If $\tilde{\Omega}_j = \emptyset$, we set $\Omega_j = \emptyset$.

Applying concave subdivision to $\Omega_j$ we get concave domains $T(\Omega_j)$, $\Omega_{j1}$, and $\Omega_{j2}$. Continuing in this manner we get the collection of convex domains $\Omega_{j_1\dots j_d}$  where $j_i \in \{1,2\}$ and $d \in \NN$. The weight expansion of $\Omega$ is the, possibly finite, multiset
$$
w(\Omega) = \{ \tau(\Omega)\} \cup \{ \tau(\Omega_{j_1\dots j_d}) \mid  j_i \in \{1,2\},\,  d \in \NN,\, \Omega_{j_1\dots j_d} \neq \emptyset\}.
$$
Ordering this multiset with repetitions, we get the ordered weight expansion of $\Omega$ 
$$\vec{w}(\Omega) =\{w_1, w_2, w_3 \dots \}$$ with $$w_1 \geq w_2 \geq \dots.$$

\begin{lemma}\label{hhh}
Suppose that $h$ is in $\widehat{\mathcal{C}}(\lambda)$ and that $h'(0)<-4$.  For the concave domain $\Omega$ bounded by the axes and the graph of $h$, we have 
\begin{eqnarray*}
\tau(\Omega) &=& y_0+ h(y_0) \\
\tau(\Omega_2)=\tau(\Omega_1) &=& 2y_2+  h(y_2) -\tau(\Omega) \\
\tau(\Omega_{22})=\tau(\Omega_{11}) &=& 3y_{22}+ h(y_{22}) -\tau(\Omega)-\tau(\Omega_2) \\
\tau(\Omega_{21})=\tau(\Omega_{12}) &=& 3y_{21}+  2h(y_{21}) -2\tau(\Omega)-\tau(\Omega_2) \\
\tau(\Omega_{222})=\tau(\Omega_{111}) &=& 4y_{222}+  h(y_{222}) -\tau(\Omega)-\tau(\Omega_2) -\tau(\Omega_{22})\\
\tau(\Omega_{221})=\tau(\Omega_{112}) &=& 5y_{221}+ 2 h(y_{221}) -2\tau(\Omega)-2\tau(\Omega_2) -\tau(\Omega_{22})\\
\tau(\Omega_{212})=\tau(\Omega_{121}) &=& 5y_{212}+ 3 h(y_{212}) -3\tau(\Omega)-2\tau(\Omega_2) -\tau(\Omega_{21})\\
\tau(\Omega_{211})=\tau(\Omega_{122}) &=& 4y_{211}+ 3 h(y_{211}) -3\tau(\Omega)-\tau(\Omega_2) -\tau(\Omega_{21})
\end{eqnarray*}
where the $y$ values are determined uniquely by the conditions 
\begin{eqnarray*}
h'(y_0) & = & -1 \\
h'(y_2) & = & -2 \\
h'(y_{22}) & = &  -3 \\
h'(y_{21}) & = &  -\frac{3}{2} \\
h'(y_{222}) & = &  -4 \\
h'(y_{221}) & = &  -\frac{5}{2} \\
h'(y_{212}) & = &  -\frac{5}{3} \\
h'(y_{211}) & = &  -\frac{4}{3} 
\end{eqnarray*}
\end{lemma}

\begin{proof}
The leftmost set of equalities all follow from the symmetry of $h$. The condition
$h'(0)<-4$ implies that all the points $y_1, \dots, y_{211}$ exist. The rest of the formulas follow easily from the concave subdivision process described above.
\end{proof}

Let  $\rho \colon [0,\lambda] \to \RR_{\geq 0}$ be a smooth function such that  
\begin{itemize}
  \item[(i)] the support of $\rho$ is contained in an arbitrarily small interval around $y_{22}$.
  \item[(ii)] $\rho =1$ in a smaller neighborhood of $y_{22}$.
  \item[(iii)] $\int_0^1 \rho =0.$
\end{itemize}
Choose   $\delta_0>0$ small enough so that the function $$h^{\delta}  =h+\delta( \rho + \tilde{\rho})$$ is in $\widehat{\mathcal{C}}(\lambda)$ for all $\delta \leq \delta_0$. We will refer to process of moving from $h$ to $h^{\delta}$ as the $\delta$-shift. Let $\Omega^{\delta}$ be the region bounded by the axes and the graph of $h^{\delta}$ and set $X_{h^{\delta}} = \mu^{-1}(\Omega^{\delta}).$ Note  that $y_0 =x(h)$ and that $y_2= \check{x}_2$. It then follows from Lemma \ref{ready} that 
\begin{equation*}
\label{ }
c_k(X_{h^{\delta}}) = c_k(X_h)  \text{ for all  $k \in \NN$  and $\delta \in [0,\delta_0]$},
\end{equation*}
and 
\begin{equation*}
\label{ }
\mathrm{volume}(X_{h^{\delta}}) = \mathrm{volume}(X_h)  \text{ for all  $\delta \in [0,\delta_0]$}.
\end{equation*}

By construction, we have the following formulas relating $\tau$-values before and  after the $\delta$-shift:
\begin{equation}\label{delta}
\begin{aligned}
\tau(\Omega^{\delta}) &=\tau(\Omega) \\
\tau(\Omega^{\delta}_2)=\tau(\Omega^{\delta}_1) &= \tau(\Omega_2) \\
\tau(\Omega^{\delta}_{22})=\tau(\Omega^{\delta}_{11}) &= \tau(\Omega_{22})+\delta\\
\tau(\Omega^{\delta}_{21})=\tau(\Omega^{\delta}_{12}) &= \tau(\Omega_{21}) \\
\tau(\Omega^{\delta}_{222})=\tau(\Omega^{\delta}_{111}) &= \tau(\Omega_{222}) -\delta \\
\tau(\Omega^{\delta}_{221})=\tau(\Omega^{\delta}_{112}) &= \tau(\Omega_{221})-\delta\\
\tau(\Omega^{\delta}_{212})=\tau(\Omega^{\delta}_{121}) &= \tau(\Omega_{212})\\
\tau(\Omega^{\delta}_{211})=\tau(\Omega^{\delta}_{122}) &= \tau(\Omega_{211})
\end{aligned}
\end{equation}

To say anything about the ordered weight expansion of $\Omega^{\delta}$ we now need to consider an explicit function $h$. We will define $h$ by parameterizing its graph. We start with the curve from Example \ref{lag},
\begin{equation*}
\label{ }
\alpha(t) =\left( 2 \sin \left(\frac{t}{2}\right)  -t \cos \left(\frac{t}{2}\right),   2 \sin \left(\frac{t}{2}\right)  + (2\pi -t) \cos \left(\frac{t}{2}\right)\right),\quad t \in [0,2\pi].
\end{equation*}
For a fixed, suitably small $\epsilon>0$, we then set $$\gamma(t) =\alpha(t)-(\epsilon, \epsilon) =(\gamma_1(t), \gamma_2(t))$$ where the domain of $\gamma$ is now $[\xi, 2\pi-\xi]$ for the number $\xi>0$ defined uniquely by the condition 
\begin{equation*}
\label{ }
2 \sin \left(\frac{\xi}{2}\right)  -\xi \cos \left(\frac{\xi}{2}\right) =\epsilon.
\end{equation*}

Let $h$ be the function defined by the image of $\gamma$. Then $h$ is in $\widehat{\mathcal{C}}(2-\epsilon)$.  By Lemma \ref{reg},  $X_{h}$ has a smooth boundary thanks to the shift by $\epsilon$.
A simple computation yields 
\begin{equation}
\label{t2y}
h'(\gamma_1(t)) = \frac{\gamma_2'(t)}{\gamma_1'(t)} = -\frac{2\pi -t}{t}
\end{equation}
and hence
\begin{equation*}
\label{ }
h'(0) =-\frac{2\pi -\xi}{\xi}.
\end{equation*}
This is finite and less than $-4$ for all sufficiently small $\epsilon>0$.

Let $\Omega$ be the domain defined by $h$.  It then follows from \eqref{t2y} and Lemma \ref{hhh} that:
\begin{eqnarray*}
\tau(\Omega) &=& 4-2\epsilon \\
\tau(\Omega_2)=\tau(\Omega_1) &=& 3\sqrt{3} -4 -\epsilon \approx 1.19615 -\epsilon\\
\tau(\Omega_{22})=\tau(\Omega_{11}) &=& 4\sqrt{2} - 3\sqrt{3} -\epsilon \approx .46070-\epsilon\\
\tau(\Omega_{21})=\tau(\Omega_{12}) &=& 10 \sin\left(\frac{3\pi}{5}\right) -3\sqrt{3} - 4 \approx .31441\\
\tau(\Omega_{222})=\tau(\Omega_{111}) &=& 10 \sin\left(\frac{\pi}{5}\right) -4\sqrt{2}-\epsilon\approx.22010-\epsilon\\
\tau(\Omega_{221})=\tau(\Omega_{112}) &=& 14 \sin\left(\frac{2\pi}{7}\right)  -3\sqrt{3} -4\sqrt{2}\approx.09263\\
\tau(\Omega_{212})=\tau(\Omega_{121}) &=& 16 \sin\left(\frac{3\pi}{8}\right)  - 3\sqrt{3} -10 \sin\left(\frac{3\pi}{5}\right) +\epsilon \approx.07535 +\epsilon\\
\tau(\Omega_{211})=\tau(\Omega_{122}) &=& 14 \sin\left(\frac{3\pi}{7}\right) -4 - 10 \sin\left(\frac{3\pi}{5}\right) +\epsilon \approx 0.13843+\epsilon
\end{eqnarray*}

For all sufficiently small $\epsilon>0$ we then have
\begin{eqnarray*}
\vec{w}(\Omega) & = & \{w_1, w_2, \dots\} \\
{} & = & \{ \tau(\Omega), \tau(\Omega_2), \tau(\Omega_1),\tau(\Omega_{22}), \tau(\Omega_{11}), \tau(\Omega_{21}),\tau(\Omega_{12}), \dots\}. 
\end{eqnarray*}
Similarly, for all sufficiently small $\epsilon>0$ and $\delta>0$ we have
%\begin{equation}
%\label{ }
%\vec{w}(\Omega^{\delta})=\{ \tau(\Omega^{\delta}), \tau(\Omega^{\delta}_2), \tau(\Omega^{\delta}_1),\tau(\Omega^{\delta}_{22}), \tau(\Omega^{\delta}_{11}), \tau(\Omega^{\delta}_{21}),\tau(\Omega^{\delta}_{12}), \dots\}
%\end{equation}
\begin{eqnarray*}
\vec{w}(\Omega^{\delta}) & = & \{ w^{\delta}_1, w^{\delta}_2, \dots\}\\
{} & = & \{ \tau(\Omega^{\delta}), \tau(\Omega^{\delta}_2), \tau(\Omega^{\delta}_1),\tau(\Omega^{\delta}_{22}), \tau(\Omega^{\delta}_{11}), \tau(\Omega^{\delta}_{21}),\tau(\Omega^{\delta}_{12}), \dots\}\\
{} & = & \{ \tau(\Omega), \tau(\Omega_2), \tau(\Omega_1),\tau(\Omega_{22})+\delta, \tau(\Omega_{11}) +\delta, \tau(\Omega_{21}),\tau(\Omega_{12}), \dots\}
 \end{eqnarray*}

To complete the proof of Theorem \ref{blind} it suffices to prove the following.

\begin{lemma}\label{capcomp}
For all sufficiently small $\epsilon>0$ and $\delta>0$ we have
\begin{equation}
\label{comp}
c^{\mathrm{ECH}}_9(X_{h+\delta (\rho+ \tilde{\rho})}) =c^{\mathrm{ECH}}_9(X_{h})+\delta.
\end{equation}
\end{lemma}

\begin{proof}
By equation \eqref{choi} we have
\begin{equation*}
\label{9}
c^{\mathrm{ECH}}_9(X_{h})= \max\left\{  \sum_{i=1}^k d_i w_i \,\middle|\, \sum_{i=1}^k \frac{d_i^2+d_i}{2} \leq 9, \, d_i \in \{0\} \cup \NN \right\}.
\end{equation*}
The two largest terms in the set 
$$
\left\{  \sum_{i=1}^k d_i w_i \,\middle|\, \sum_{i=1}^k \frac{d_i^2+d_i}{2} \leq 9, \, d_i \in \{0\} \cup \NN \right\}
$$
are $3w_1+2w_2$ and $3w_1 +w_2 +w_3 +w_4.$ For us, $w_2=w_3$ and $w_4>0$ and so 
\begin{eqnarray*}
c^{\mathrm{ECH}}_9(X_{h}) & = & 3w_1 +w_2 +w_3 +w_4 \\
 {} & = & 3\tau(\Omega) +\tau(\Omega_2)+\tau(\Omega_1)+\tau(\Omega_{22}). 
\end{eqnarray*}
An identical argument then yields
\begin{eqnarray*}
c^{\mathrm{ECH}}_9(X_{h+\delta (\rho+ \tilde{\rho})}) & = & 3\tau(\Omega^{\delta}) +\tau(\Omega^{\delta}_2)+\tau(\Omega^{\delta}_1)+\tau(\Omega^{\delta}_{22}) \\
{} & = & 3\tau(\Omega) +\tau(\Omega_2)+\tau(\Omega_1)+\tau(\Omega_{22}) +\delta \\
{} & = &  c^{\mathrm{ECH}}_9(X_{h}) + \delta,
\end{eqnarray*}
as required.
\end{proof}

\section{Further Questions}\label{open}

Here we discuss some unresolved questions motivated by Theorems  \ref{novolume}, \ref{mutual} and \ref{blind} and their proofs, and describe some relevant examples.

\subsection{On varying capacities one at a time.} 
The proof of Theorem \ref{mutual} requires the consideration of both convex and concave toric domains to allow for the independent variation of both even and odd index capacities.

\begin{question}
Does there exist a star-shaped domain $U \subset \RR^{2n}$, such that for any integer $j$, there is star-shaped domain $V_j \subset \RR^{2n}$ with
\begin{equation*}
\label{ }
c_k(V_j) =c_k(U)  \,\,\,\,\text{  for all  } k\neq j
\end{equation*}
and 
\begin{equation*}
\label{ }
c_j(V_j) \neq c_j(U)?
\end{equation*}
\end{question}

It is also not clear from the proof of Theorem \ref{mutual} whether a fixed capacity can be independently varied through the full range of its possible values while keeping the other capacities fixed.

\begin{question}
Does there exist a star-shaped domain $U \subset \RR^{2n}$ with $c_{j-1}(U) < c_{j+1}(U)$ for some $j\geq 2$, such that for each $a \in (c_{j-1}(U), c_{j+1}(U))$ there is star-shaped domain $V_a$ with
\begin{equation*}
\label{ }
c_k(V_a) =c_k(U)  \,\,\,\,\text{  for all  } k\neq j
\end{equation*}
and 
\begin{equation*}
\label{ }
c_j(V_a) = a?
\end{equation*}
\end{question}

\subsection{Isocapacity variations of volume}

Let $U$ is a star-shaped domain in  $\RR^{2n}$. Define the \emph{isocapacity volume ratio} of $U$ to be 
$$\mathrm{IVR}(U) = \sup \frac{\mathrm{volume}(V)}{\mathrm{volume}(W)}$$
where the supremum is taken over all star-shaped domains $V, W \subset \RR^{2n}$ such that $c_k(V) =c_k(W) = c_k(U)$ for all $k \in \NN$.

\begin{example}
For any star-shaped domain $U \subset \RR^{2}$  we have $\mathrm{IVR}(U) =1$
\end{example}

\begin{example}
For any symplectic polydisk $P \subset \RR^{2n}$ with $n>1$ we have $\mathrm{IVR}(P) =\infty$.
\end{example}

With these simple cases in mind we restrict our attention again to the case of star-shaped domains, with smooth boundaries, in $\RR^{2n}$ for $n>1$. The proof of Theorem \ref{novolume} provides a mechanism, in the strictly convex toric setting,  for varying volumes while keeping capacities fixed. This mechanism fails in the case of symplectic ellipsoids, since the usable gaps between capacity carriers vanish as one relaxes strict convexity. This observation motivates the following.

\begin{question}
Is the isocapacity volume ratio of any symplectic ellipsoid equal to $1$?
\end{question}

The general question concerning isocapacity volume ratios is the following.

\begin{question}
Is the isocapacity volume ratio of every star-shaped domain $U$ in $\RR^{2n}$ with smooth boundary finite? If so, is there some universal upper bound, dependent on the dimension, for the isocapacity volume ratio of all star-shaped domains in $\RR^{2n}$ with smooth boundary?
\end{question}

\begin{example}[Convex graphs for $n=2$.]
\label{circle}
Let $f$ be a function in $\widehat{\mathcal{V}}(\lambda)$ such that $f'(0) \in (-1/2,0)$.
Here, we use Corollary \ref{f} to obtain a lower bound for $\mathrm{IVR}(X_f)$. It is not clear whether this bound can be improved.

We first construct a function $\underline{f} \colon [0,\lambda] \to \RR$ whose graph lies below that of $f$, such that the domain $X_{\underline{f}}$ has the same capacities as $X_f$.  Define $\underline{f} $ to be the piecewise linear function obtained by first connecting $(x_{k-1}, f(x_{k-1}))$ to $(x_{k},f(x_k)))$ for all odd $k \in \NN$, where $x_0 = \lambda$, and then extending this to $[x(f),\lambda]$ as a symmetric function. 
Note that for any $\underline{\epsilon}>0$ there is a function $f_{\underline{\epsilon}}$ in $\widehat{\mathcal{V}}(\lambda)$ such that $\|\underline{f} -f_{\underline{\epsilon}}\|_{C^0} < \underline{\epsilon}$,
\begin{equation*}
\label{ }
f_{\underline{\epsilon}}(x_k) = f(x_k),
\end{equation*}
and
\begin{equation*}
\label{ }
f_{\underline{\epsilon}}'(x_k) = -\frac{k-1}{k+1}
\end{equation*}
for all odd $k \in \NN$. It follows from this, and the continuity of the $c_k$, that 
$$c_k(X_{\underline{f}}) = c_k(X_f)$$ for all $k \in \NN$.  
%By construction we also have 
%$$\mathrm{volume}(X_{\underline{f}}) < \mathrm{volume}(X_f).$$

Next we construct a function  $\bar{f} \colon [0,\lambda] \to \RR$ whose graph lies above that of $f$ such that the domain $X_{\bar{f}}$ has the same capacities as $X_f$. For odd $k$, let $L_k$ be the tangent line to the graph of $f$ at $(x_k, f(x_k))$. Let $L_0$ be the line through $(0,\lambda)$ with slope $f'(0)$. Denote the intersection point of $L_{k-1}$ and $L_k$ by $p_k$. Let $\bar{f} \colon [0,\lambda] \to \RR$ be the piecewise linear function obtained by first connecting $p_{k-1}$ to $p_{k}$ for all odd $k \in \NN$,  and then extending this to $[x(f),\lambda]$ as a symmetric function. 
%Note that for any $\bar{\epsilon}>0$ there is a function $f_{\bar{\epsilon}}$ in $\widehat{\mathcal{V}}(\lambda)$ such that $\|\bar{f} -f_{\bar{\epsilon}}\|_{C^0} < \epsilon$,
%\begin{equation}
%\label{ }
%f_{\bar{\epsilon}}(x_k) = f(x_k),
%\end{equation}
%and
%\begin{equation}
%\label{ }
%f_{\bar{\epsilon}}'(x_k) = -\frac{k-1}{k+1}
%\end{equation}
%for all odd $k \in \NN$. It follows from this, and the continuity of the $c_k$, that 
Arguing as above,  we have $$c_k(X_{\bar{f}}) = c_k(X_f)$$ for all $k \in \NN.$ 

By construction, we have $\underline{f} \leq f \leq \bar{f}$  and hence  $X_{\underline{f}} \subset X_f \subset X_{\bar{f}}$. From this it follows that  
\begin{equation}
\label{updown}
\mathrm{IVR}(X_f) \geq \frac{\mathrm{volume}(X_{\bar{f}})}{\mathrm{volume}(X_{\underline{f}})}.
\end{equation}
Since both $\bar{f}$ and $\underline{f}$ are piecewise linear, the volumes on the right are easily computable. For example, for the function $f_2(x) = \sqrt{1-x^2}$ from Example \ref{round},  it follows from \eqref{updown} that $\mathrm{IVR}(X_{f_2})$ is at least
\begin{equation*}
\label{ }
\frac{\sqrt{5}-2+\sum_{k=1}^{\infty}\left[\sqrt{k^2+(k+1)^2}\left( \sqrt{(k-1)^2+k^2}+\sqrt{(k+1)^2+(k+2)^2}  \right) -2 (k^2+(k+1)^2)\right]}{\sum_{k=1}^{\infty} (4k^4+1)^{-\frac{1}{2}}}
\end{equation*}
or, approximately, $1.0335$.

\begin{question}
Is inequality  \eqref{updown} really an equality?
\end{question}

\end{example}

\begin{example}[Toric polytopes in $\RR^4$]
Here we use the capacity formula for symmetric toric polytopes from Proposition \ref{poly},  to find simple lower bounds for  the isocapacity volume ratios of the family of simple star-shaped domains in $\RR^4$ from Example \ref{diagonal}. For $r \in [2/3,1)$ let $\Omega_r$ be the convex hull of $\{(0,0), (1,0), (0,1), (r,r)\}$. For $a$, $b>0$  let $\Omega_{ab}^r$  be the convex hull of $$\{(0,0), (1,0), (0,1), (r,r), (a,b), (b,a)\}.$$ Set $X_r = \mu^{-1}(\Omega_r)$ and $X_{ab}^r = \mu^{-1}(\Omega_{ab}^r)$.

It is straight forward to check that the symmetric toric domain $X_{ab}^r$ is convex for all $(a,b)$ in the region 
\begin{equation*}
\label{ }
I_r=\{(a,b) \mid 0 \leq b \leq a \leq 1,\, a+b \leq 2r\}.
\end{equation*}

\begin{lemma}
The domains $X_r$ and $X_{ab}^r$ have the same capacities for all $(a,b)$ in  the region
\begin{equation*}
\label{ }
J_r = I_r \cap \{(a,b) \mid  2a+b \leq 3r\}.
\end{equation*}

\end{lemma}

\begin{proof}
It follows from Proposition \ref{poly}, and our choice of $r\geq 2/3$,  that for $(a,b)$ in $I_r$ we have  
\begin{equation}
\label{iso}
c_k(X_{ab}^r) = c_k(X_r)
\end{equation}
for all $k \in \NN$ if and only if
\begin{equation*}
\label{ }
(n+1)a + nb \leq (2n+1)r
\end{equation*}
for all $n \in \NN$. Given the condition $a+b \leq 2r$ from $I_r$, it is sufficient to require the first of these inequalities, 
\begin{equation*}
\label{ }
2a + b \leq 3r.
\end{equation*}
\end{proof}

The volume of $X_{ab}^r$ is equal to $2\pi(r(a-b)+b))$.  This function of $(a,b)$ takes its maximum value on $J_r$, $
2\pi(6r-3r^2-2),$ at the boundary point $(1,3r-2).$ 
Hence,
\begin{equation}
\label{ivr}
\mathrm{IVR}(X_r) \geq 3(2-r) -\frac{2}{r}.
\end{equation}
The maximum lower bound for the isocapacity volume ratios attained in this way is $6-2\sqrt{6} \approx 1.10102$ at $r=\sqrt{\frac{2}{3}}.$

%\begin{remark}
%The domains $X_r$ do not have smooth boundaries. However, they still belong to the current discussion for the following reason. Arguing as in Example \ref{circle}, with the boundary of $\Omega_r$ playing the role of the graph of $\underline{f}$, one can approximate each $X_r$ in the $C^0$-topology  by a  domain with a smooth boundary and the same capacities as $X_r$.
%\end{remark}

\begin{remark}
For $r \in [1/2, 2/3]$ the lower bounds for $\mathrm{IVR}(X_r)$ are smaller and more difficult to derive since the conditions for the isocapacity property are more complicated. This is the reason that these allowable values of $r$ are disregarded in the discussion above.
\end{remark}

\begin{remark}
Adding more than two vertices to the original polytope $\Omega_r$ does not improve the lower bound in \eqref{ivr}. It would be interesting to know if inequality \eqref{ivr} is really an equality. If this were true then the function $r \mapsto \mathrm{IVR}(X_r)$ would approach $1$ as $r \to 1^-$, whereas $\mathrm{IVR}(X_1) = \mathrm{IVR}(P(1,1)) =\infty.$
\end{remark}

\begin{remark}
The authors have not found domains, including symmetric toric polytopes in higher dimensions, which have finite isocapacity volume ratios larger than $6-2\sqrt{6}$. 
\end{remark}
\end{example}

\end{document}